\documentclass[11pt]{amsart}
\usepackage{fullpage}
\usepackage{microtype}
\usepackage{graphicx}

\usepackage[colorlinks]{hyperref}
\usepackage[dvipsnames]{xcolor}
\usepackage{stmaryrd}

\usepackage{framed}

\hypersetup{
    linkcolor=red,
    citecolor=OliveGreen,
    filecolor=black,
    urlcolor=black,
}
\usepackage{caption}
\usepackage{upgreek}

\usepackage{booktabs}

\newcommand{\Adv}{\mathsf{A}}

\renewcommand{\sf}{f}

\newcommand{\cP}{\mathscr{P}}

\newcommand{\cF}{\mathscr{L}}
\newcommand{\cE}{\mathscr{E}}
\newcommand{\cH}{\mathscr{H}}

\newcommand{\cI}{\mathcal{I}}
\newcommand{\cJ}{\mathcal{J}}

\newcommand{\Supp}{\mathsf{Supp}}

\newcommand{\se}{\mathsf{e}}
\newcommand{\sh}{\mathsf{h}}
\newcommand{\sg}{\mathsf{g}}
\newcommand{\sm}{\mathsf{m}}

\newcommand{\sch}{\mathsf{s}}

\newcommand{\sLambda}{\mathsf{\Lambda}}

\newcommand{\Par}{\mathsf{Par}}

\newcommand{\contains}{\supseteq}

\newcommand{\sM}{\mathsf{M}}

\renewcommand{\Re}{\operatorname{\mathsf{Re}}}
\renewcommand{\Im}{\operatorname{\mathsf{Im}}}

\newcommand{\beq}[1]{\begin{equation}\label{eq:#1}}
\newcommand{\eeq}{\end{equation}}

\newcommand{\sign}{\text{sign}}

\usepackage{verbatim}
\usepackage{amsmath}
\usepackage{amsthm}
\usepackage{amssymb}
\usepackage{bm}
\usepackage{bbm}
\usepackage{dsfont}
\usepackage{comment}
\usepackage{multirow}
\usepackage{color}
\usepackage{tikz}
\usepackage{mathtools}
\usepackage{pifont}
\usepackage{enumerate,algorithm,algorithmic,caption,graphicx,color}

\usepackage{chngcntr}

\newcommand{\be}{\begin{eqnarray}}
\newcommand{\ee}[1]{\label{#1}\end{eqnarray}}

\newcommand{\ese}{\end{eqnarray*}}
\newcommand{\bse}{\begin{eqnarray*}}

\def\qed{\hfill$\square$}

\def\Conv{\mathop{\hbox{\rm Conv}}}

\def\N{{\mathbb{N}}}

\def\sF{{\mathsf{F}}}
\def\sQ{{\mathsf{Q}}}
\def\C{\mathbb{C}}
\def\R{\mathbb{R}}
\def\Z{\mathbb{Z}}
\def\Q{\mathbb{Q}}

\newcommand{\wt}[1]{{\widetilde{#1}}}
\def\Argmax{\mathop{\hbox{\rm Argmax}}}

\def\T{{\mathbb{T}}}
\def\D{{\mathbb{D}}}

\def\vphi{\varphi}

\newcommand{\bxi}{\boldsymbol{\xi}}
\newcommand{\bpsi}{\boldsymbol{\psi}}

\newtheorem{theorem}{Theorem}
\newtheorem{proposition}{Proposition}[section]
\newtheorem{lemma}{Lemma}[section]
\newtheorem{fact}{Fact}[section]
\newtheorem{definition}{Definition}[section]

\newtheorem{example}{Example}[section]
\newtheorem{corollary}{Corollary}[section]
\newtheorem{conjecture}{Conjecture}[section]
\newtheorem{remark}{Remark}[section]

\newcommand{\veps}{\varepsilon}

\newcommand{\ones}{1}

\renewcommand{\le}{\leqslant}
\renewcommand{\ge}{\geqslant}

\newcommand{\negphantom}[1]{\ifmmode\settowidth{\dimen0}{$#1$}\else\settowidth{\dimen0}{#1}\fi\hspace*{-\dimen0}}

\allowdisplaybreaks

\usepackage{xfrac}
\usepackage{mathrsfs}
\usepackage{eucal}

\renewcommand{\top}{\mathsf{T}}
\newcommand{\htop}{\mathsf{H}}

\newcommand{\XX}{X}

\renewcommand{\SS}{\boldsymbol{U}}
\newcommand{\RR}{\boldsymbol{Z}}

\newcommand{\bald}{\begin{aligned}}
\newcommand{\eald}{\end{aligned}}

\usepackage{trimclip}

\makeatletter
\DeclareRobustCommand{\lasymp}{\lg@asymp{<}}
\DeclareRobustCommand{\gasymp}{\lg@asymp{>}}

\newcommand{\under@asymp}[1]{\clipbox{0pt 0pt 0pt {0.5\height}}{$\m@th#1\asymp$}}
\newcommand{\lg@asymp}[1]{\mathrel{\mathpalette\lg@asymp@{#1}}}
\newcommand{\lg@asymp@}[2]{%
  \vcenter{%
    \offinterlineskip
    \m@th
    \ialign{%
      \hfil##\hfil\cr
      $#1#2$\cr
      \under@asymp{#1}\cr
    }%
  }%
}
\makeatother

\begin{document}

\title{Amplitude maximization in stable systems, Schur positivity, and some conjectures on polynomial interpolation}
\author{Dmitrii M.~Ostrovskii\;}
\address{Georgia Institute of Technology, School of Mathematics \& H. Milton Stewart School of Industrial and Systems Engineering (ISyE), Atlanta, GA 30332, USA}
\email{ostrov@gatech.edu}
\author{\;Pavel S.~Shcherbakov}
\address{Institute of Control Sciences, Russian Academy of Sciences, Moscow, Russia, 117997}
\email{sherba@ipu.ru}

\begin{abstract}
For $r > 0$ and integers $t \ge n > 0$, we consider the following problem: maximize the amplitude $|x_t|$ at time $t$, over all complex solutions $x = (x_0, x_1, \dots)$ of arbitrary homogeneous linear difference equations of order $n$ with the characteristic roots in the disc $\{z \in \mathbb{C}: |z| \le r\}$, and with initial values $x_0, \dots, x_{n-1}$ in the unit disc. 
We find that for any triple $t,n,r$, the maximum is attained with {\em coinciding roots} on the boundary circle; in particular, this implies that the peak amplitude $\sup_{t \ge n} |x_t|$ can be maximized explicitly by studying {\em a unique} equation whose characteristic polynomial is $(z-r)^n$.
Moreover, the optimality of the cophase root configuration holds for origin-centered polydiscs.
To prove this result, we first reduce the problem to a certain interpolation problem over monomials, then solve the latter by leveraging the theory of symmetric functions and identifying the associated Schur positivity structure.
We also discuss the implications for more general Reinhardt domains.
Finally, we study the problem of estimating the derivatives of a real entire function from its values at $n/2$ pairs of complex conjugate points in the unit disc. 
We propose conjectures on the extremality of the monomial $z^n$, and restate them in terms of Schur polynomials.
\vspace{-0.92cm}
\end{abstract}

\maketitle


\section{Introduction}
\label{sec:problem}
Denoting with~$\N_0$ the set of nonnegative integers, let~$\C^{\N_0}$ be the vector space (over~$\C$) of complex sequences~$x = (x_0, x_1, \dots)$, and define the left shift operator~$\Adv: \C^{\N_0} \to \C^{\N_0}$ such that~$(\Adv x)_{t} = x_{t+1}$ for all~$t \in \N_0$.
Any ~$(f_0, f_1, \dots, f_n) \in \C^{n+1}$ with~$f_n = 1$ specifies the homogeneous difference equation
\begin{equation}
\label{def:recurrence}
\sf(\mathsf{\Adv}) x = 0
\end{equation}
where~$f(z) := \sum_{k = 0}^n f_{k} z^k$, a monic polynomial of degree~$n$, is {\em the characteristic polynomial of~\eqref{def:recurrence}.} 
Solution set of~\eqref{def:recurrence} is a {shift-invariant} subspace of~$\C^{\N_0}$, i.e.~an eigenspace of~$\Adv$ as a linear operator, and to specify its element---a particular solution---one has to assign the initial values~$(x_0, \dots, x_{n-1})$.
Naturally, one might want to study the properties of the whole solution set, or of its sufficiently large subset, rather than those of a specific solution.
This task might be challenging if one considers a whole class of equations instead of a single one, due to the highly nonlinear fashion in which a solution to~\eqref{def:recurrence} depends on the roots~$z_1, \dots, z_n$ of characteristic polynomial, in contrast to its linear dependence on the initial data.
As such, it is not surprising that the classical theory of linear systems does not say much about the {\em worst-case} behavior of solutions to~\eqref{def:recurrence} in the setup where~$z_1, \dots, z_n$ {vary} over a subset of~$\C^n$.
In this work, we give a rather striking example of such a characterization, 
whose underlying nature is algebraic, and whose existence seems to be unnoticed.
Before we proceed to the presentation of our results, let us briefly review what the classical theory has to offer.
Recall that~\eqref{def:recurrence} is called {\em stable} if its arbitrary solution satisfies~$\sup_{t \in \N_0} |x_{t}| < \infty$, and is called~{\em asymptotically stable} if~$|x_t| \to 0$ as~$t \to +\infty$. 
The textbook criterion of asymptotic stability is for~$\sf$ to be {\em Schur stable}, i.e.~have all roots in the open unit disc; in the case of usual stability, one is also allowed to have roots on the boundary circle, provided that these are {distinct}.
Next, one can provide {\em quantitative} restrictions on the asymptotic behavior of solutions.
For example,~{\em $r$-stability},
\begin{equation*}
\max \{ |z_1|, \dots, |z_n| \} \le r < 1,
\end{equation*}
implies that asymptotically as~$t \to +\infty$, solutions to~\eqref{def:recurrence} shrink exponentially fast.
More precisely,
\[
\limsup_{t \to +\infty} \; \log |x_t| - t\log{r} - (n-1)\log{t} < +\infty
\] 
where the linear in~$t$ term is defined solely by the magnitude of the largest root, whereas the logarithmic term accounts for possible locations of the roots and is realized when they are identical.

We focus on a very natural question one could ask in the context of stability, phrased as follows:
\begin{quote}
\label{def:question}
{\em Given any~$r > 0$ and~$t \ge n > 0$, what is the largest possible amplitude~$|x_t|$ for a solution~$x$ to an~$r$-stable equation~\eqref{def:recurrence} of order~$n$, provided that~$|x_t| \le 1$ for~$0 \le t < n$?}
\end{quote}
\vspace{0.1cm}
In this ``vanilla'' version of the {amplitude maximization} problem, both the characteristic roots and the initial values are bounded {\em uniformly}, and one might ask how reasonable are these two assumptions, especially the latter one. 
In addition to its adoption in prior work~\cite{izmailov1987peak,sussmann2002peaking,polyak2002superstable,polyak2016large,polyak2018peak,smirnov2024asymptotic}, the assumption of uniformly bounded initial values can be justified as the one allowing for most variation at the onset of a trajectory.
Yet, our main result holds in a more general setup, allowing for heterogeneous bounds. 
To state it, we need some notation.
Let~$\cP_n$ be the space of all monic polynomials of degree~$n$, and let~$P_n(\RR)$ be the set of monic polynomials with the tuple of roots from~$\RR \subseteq \C^n$.
Let~$\XX(f)$ be the solution set of~\eqref{def:recurrence} with a given~$f$.
For~$t \in \N_0$,~$f \in \cP_n$, and~$\SS \subseteq \C^n$,
\begin{align}
\label{def:instant-amplitude}
M_{t}(f|\SS) &:= \sup \, \{\, |x_t|: \; x \in \XX(f),\;\; (x_{0},\dots, x_{n-1}) \in \SS \,\}
\end{align}
is the {\em maximum amplitude at moment~$t$ of solutions to~\eqref{def:recurrence}}, with initialization in~$\SS$. 
Furthermore,
\begin{align}
\label{def:max-instant-amplitude}
\sM_{t}(\RR|\SS) &:= \sup_{f \in P_n(\RR)} M_{t}(f|\SS)
\end{align}
is the maximum amplitude (at~$t$) for the associated with~$\RR$ {\em class of equations}.
Finally, we define
\begin{align}
\label{def:max-peak-amplitude}
\sM_*(\RR|\SS) := \sup_{t \in \Z:\, t \ge n} \sM_{t}(\RR|\SS),
\end{align}
the maximum {\em peak amplitude} in the class of equations associated to~$\RR$, with initialization over~$\SS$. 
The basic version of the problem corresponds to~$(\RR,\SS) = (\D(r)^n, \D(1)^n)$, where~$\D(r)$ is the closed origin-centered disc of radius~$r$.
Meanwhile, our main result allows for each domain to be a {polydisc}
\begin{equation*}
\D_n(r_1,\dots,r_n) := \D(r_1) \times \dots \times \D(r_n)
\end{equation*}
with arbitrary radii~$r_1, \dots, r_n$ that are allowed both to vanish and to exceed~$1$. Let us now state it.


\begin{theorem}
\label{th:main-result}
For integer~$t \ge n > 0$, functional~\eqref{def:max-instant-amplitude} evaluated on arbitrary pair of polydiscs satisfies
\begin{align}
\label{eq:main-result}
\sM_{t}(\D_n(r_1,\dots,r_n)|\D_n(w_1,\dots,w_n)) 
&= M_{t} (\textstyle\prod_{k = 1}^n (z-r_k e^{i\theta})|\D_n(w_1,\dots,w_n)) \quad \forall \theta \in \R\\
&=\sum_{k \in [n]} w_k \sch_{(t-n|n-k)}(r_1,\dots,r_n); \notag
\end{align}
here~$\sch_{(a|b)}$ is the Schur polynomial of the hook~$(a|b) = (1_{a+1},1^b)$.
Thus,~$|x_t|$ is maximized with cophase roots.
The supremum in the right-hand side of~\eqref{eq:main-result} is attained on~$x_{k-1} = -w_{k} e^{i(\pi-\theta)(n-k+1)}$,~$k \in [n]$.
\end{theorem}
%
%
Deferring the discussion of the Schur functions to~Section~\ref{sec:partitions}, let us explore the implications of~\eqref{eq:main-result}.

First of all, observe that according to Theorem~\ref{th:main-result},~$\sM_t$ is delivered by the {\em same} set of optimal pairs~$(f; (x_{0},\dots,x_{n-1}))$, regardless of~$t$. 
As the result, computing the maximum peak amplitude~\eqref{def:max-peak-amplitude} for polydiscs reduces to computing the maximum amplitudes for~$t \ge n$.
In fact, the peak amplitude is infinite if~$\max\{r_1, \dots, r_n\} > 1$, yet Theorem~\ref{th:main-result} still holds and gives the maximum amplitude {at~$t$}.

Second, note that maximization over~$\SS$ is easy: once~$f \in P_n(\RR)$ is fixed,~$x_t$ becomes a linear functional.
The actual challenge is to maximize over~$\RR$ and establish that the cophase configuration is optimal.
To do so, we proceed in two steps. 
In the first step, we reformulate~$M_t(f|\SS)$, cf.~\eqref{def:instant-amplitude}, as the sum of absolute values of the coefficients for a certain interpolating polynomial, with~$z_1,\dots,z_n$ as the nodes,
and express the coefficients explicitly as homogeneous symmetric polynomials in~$z_1, \dots, z_n$. 
In the second step, we establish the Schur positivity of these polynomials, which implies the result.
(If unfamiliar with the concept, the reader will understand this implication upon consulting Section~\ref{sec:partitions}.)



Finally, let us mention that Theorem~\ref{th:main-result} can be applied in the more general scenario where~$\RR$ and~$\SS$ are origin-centered {Reinhardt domains}, i.e.~are closed under entrywise phase modulation~\cite{vladimiroff1966methods,hormander1973introduction,shabat1992introduction}. 
In this case, Theorem~\ref{th:main-result} gives
$
\sM_t(\RR|\SS) = \sM_t(\RR_+|\SS_+),
$
where~$\mathbf{D}_+$ is the ``radius hull'' of~$\mathbf{D} \in \C^{n}$, obtained by taking the entrywise moduli of all points. 
For {logarithmically convex} Reinhardt domains, evaluating~$\sM_t(\RR_+|\SS_+)$ amounts to solving a bilevel optimization problem with a special structure, and we demonstrate how to exploit this structure when explicit evaluation is out of the question.
In particular, we give an algorithm to compute~$\sM_t(\RR_+|\SS_+)$ in the setup where~$\RR_+$ is log-affine, and~$\SS_+$ admits a linear maximization oracle.
We also give some examples of~$(\RR_+, \SS_+)$ allowing for explicit characterizations of the maximizing pair, in the vein of the characterization in Theorem~\ref{th:main-result}. 

\subsection*{Contributions and outline of the paper.}
In Section~\ref{sec:warmup}, we carry out a preliminary study of the amplitude maximization problem. First, we reduce it to a certain polynomial interpolation problem -- namely, Lagrange interpolation of monomials in the nodes~$z_1, \dots, z_n$, with error measured by the weighted sum of absolute coefficients of the Lagrange polynomial. 
We then compute~$\sM_t(\D(r)^n|\D(1)^n)$ explicitly as a function of~$t,n,r$ (assuming Theorem~\ref{th:main-result} to be valid).
Finally, we present some preliminary results with elementary proofs, and shed light on the arising challenges.
%
Section~\ref{sec:partitions} gives background on symmetric functions and partitions, needed to prove Theorem~\ref{th:main-result}. 
The proof is presented in Section~\ref{sec:result}, along with an alternative calculation of~$\sM_t(\D(r)^n|\D(1)^n)$, this time using combinatorial tools. 
%
Section~\ref{sec:reinhardt} explores the implications of Theorem~\ref{th:main-result} for Reinhardt domains, alluded to in the previous paragraph.
%
%
In Section~\ref{sec:shadrin}, we apply our techniques to study the problem of approximating the derivatives of a real entire function by those of its Lagrange polynomial. In particular, we consider the case of pairwise conjugate nodes (real characteristic polynomial), propose some conjectures on the worst-case error, and then reformulate these in terms of Schur polynomials.
\subsection*{Notation.}
\label{sec:notation}
For integer~$n \ge 0$, we define~$\N_n  := \{n, n+1, \dots\}$ and~$[n] := \N_1 \setminus \N_{n+1}$.
For any~$k \in [n]$, we abbreviate~$[n] \setminus \{k\}$ to~$[n] \setminus k$.
For integers~$m \le n$,  we write~$(a_m, \dots, a_n) \in \C^{n-m+1}$ as~$a_{m:n}$, and we use the abbreviation~$g(x_{m:n})$ for the vector~$(g(x_m),\dots,g(x_n)) \in \C^{n-m+1}$ with arbitrary~$g: \C \to \C$.
The vector with~$n$ ones is~$\ones_n$. For~$p \in [1,+\infty]$, we let~$\|\cdot\|_{p}$ be the~$p$-norm on~$\C^n$. 
$\Conv(\boldsymbol{X})$ is the convex hull of~$\boldsymbol{X} \subset \R^n$. 
We let~$\R_+^n$ be the nonnegative orthant, and~$a_{1:n} \ge b_{1:n}$ means that~$a_{1:n} - b_{1:n} \in \R_+^{n}$. 
We let~$A^\top$ and~$A^\htop$ be the ordinary and Hermitian transposes of~$A \in \C^{m \times n}$. 
$\T(r)$ and~$\D(r)$ are the origin-centered circle and~{closed} disk of radius~$r$;~$\T,\D$ are the unit ones.
We define the torus
$
\T_n(r_{1},\dots,r_n) := \T(r_1) \times \dots \times \T(r_n).
$
We write the polynomial~$\prod_{j \in [n]} (\cdot-z_j)$ as~$P_n[z_{1:n}]$.
Letting~$D_n(z_{1:n})$ be the diagonal matrix with~$z_{1:n}$ on the diagonal, define the Vandermonde matrix
\begin{equation*}
V_n(z_{1:n}) 
:= 
\begin{pmatrix}
1      & 1      & \cdots & 1 \\
z_1    & z_2    & \cdots & z_n \\
\vdots & \vdots & \ddots & \vdots \\
z_1^{n-1} & z_2^{n-1} & \cdots & z_n^{n-1}
\end{pmatrix},
\end{equation*}
which is nonsingular if~$z_{1:n}$ is {\em simple}, i.e.~has all distinct entries.
The {\em companion matrix} of~$f \in \cP_n$ is
\begin{equation}
\label{def:companion-matrix}
A_n(f) := 
\left(
\begin{array}{c|ccc}
0\\
\vphantom{\int_0^1}\smash{\vdots} & & I_{n-1} & \\
0\\
\hline 
-\sf_{0} & -\sf_{1} & \cdots & -\sf_{n-1}
\end{array}
\right).
\end{equation}
Using it allows to vectorize~\eqref{def:recurrence} by passing to the first-order vector recurrence with respect to~$\xi^{(t)} \in \C^n$,
\begin{equation}
\label{def:recurrence-vectorized}
\xi^{(t+1)} = A_n(f) \xi^{(t)}, \quad t \in \N_0,
\end{equation}
linked to~\eqref{def:recurrence} via~$\xi^{(t)} = x_{t\,:\,t+n-1}$.
Note that, omitting~$z_{1:n}$ for brevity, one has~$A_n(P_n) = V_n D_n V_n{}^{-1}$.

\section{Preliminary results}
\label{sec:warmup}
In this section, we conduct a preliminary examination of the amplitude maximization problem. 
We start with a reduction of the amplitude maximization problem to the interpolation of monomials, which is likely known.
The subsequent subsections contain some preliminary results and discussions; these are not used in the proof of Theorem~\ref{th:main-result}, but contextualize it.
These results are obtained using elementary analytical tools, which highlights both the usefulness and the limitations of such tools.\vspace{-0.1cm}


\subsection{Reduction to interpolation of monomials}
For any simple grid~$z_{1:n}$ and tuple~$v_{1:n} \in \C^n$, there is a unique polynomial of degree~$\le n-1$ with the values~$v_{1:n}$ on~$z_{1:n}$ -- the {\em Lagrange polynomial} for the data~$(z_{1:n},v_{1:n})$.
Recall that the coefficients of this polynomial ($0^\textup{th}$ to~$n-1^\textup{st}$) are given by
\begin{equation}
\label{fact:Lagrange-via-Vandermonde}
L_n[v_{1:n}|z_{1:n}] = V_n(z_{1:n})^{-\top} v_{1:n},
\end{equation}
since the~$m^{\textup{th}}$ row of~$V_n(z_{1:n})$ lists the values of the monomial~$z^{m-1}$ on~$z_{1:n}$. 
Thus, for any~$g \in \cP_{n-1}$,
\[
\bald
(\xi^{(t)})^\top g_{0:n-1} 
\hspace{-0.05cm}=\hspace{-0.05cm} (\xi^{(0)})^\top V_n(z_{1:n})^{-\top} D_n(z_{1:n})^{t} V_{n}(z_{1:n})^{\top} g_{0:n-1} 
\hspace{-0.05cm}&=\hspace{-0.05cm} (\xi^{(0)})^\top L_n[D_n(z_{1:n})^{t} V_{n}(z_{1:n})^{\top} g_{0:n-1} | z_{1:n}],
\eald
\]
where the interpolated values~$D_n(z_{1:n})^{t} V_{n}(z_{1:n})^{\top} g_{0:n-1}$ are those of~$z^{t} g(z)$. 
Picking~$g(z) = 1$, we get

\begin{fact}
\label{fact:interpolation}
For~$t \in \N_n$, the solution to~\eqref{def:recurrence} with~$f = P_n[z_{1:n}]$ and initialization~$x_{0:n-1} \in \C^n$ satisfies
\[
x_{t} = \psi_{t}{}^{\top} x_{0:n-1}
\]
where~$\psi_{t} := L_n[(z_{1:n})^t|z_{1:n}{}^{\vphantom m}]$ are the coefficients of the degree~$n-1$ polynomial interpolating~$z^{t}$ on~$z_{1:n}$.
Note that in the case of~$t = n$, the interpolating polynomial is explicitly given by~$\psi_n(z) = z^n - f(z)$.
\end{fact}

\subsection{Explicit computation with identical roots.}
\label{sec:uniform}
Our first result is an exlicit formula for the maximum amplitudes in the case of {\em identical} characteristic roots and uniformly bounded initialization.
Recall the divided-differences form of the Lagrange polynomial (defined in~\eqref{fact:Lagrange-via-Vandermonde} in the vector form):
\begin{equation}
\label{def:Lagrange-explicit}
L_n[v_{1:n}|\zeta_{1:n}] = \sum_{k \in [n]} v_k \prod_{j \in [n] \setminus k} \frac{\zeta-\zeta_j}{\zeta_k - \zeta_j}. 
\end{equation}

\begin{proposition}
\label{prop:repeated-roots-explicit}
Let~$f_{n,r,\theta}(z) = (z-re^{i\theta})^n$ with~$r > 0$ and~$\theta \in \R$. 
For any~$t \in \N_n$, it holds that
\begin{equation}
\label{eq:repeated-roots-explicit}
M_{t}(f_{n,r,\theta}|\D^n)
= \binom{t}{n} r^{t-n} \sum_{k \in [n]} {n \choose k} \frac{k}{t - n + k} r^{k}
= \binom{t}{n} \int_{0}^r n u^{t-n}(1 + u)^{n-1}   du.
\end{equation}
In particular, one has~$M_{t}(f_{n,r,\theta}|\D^n) \le \binom{t}{n}r^{t-n}[(1+r)^n-1]$, where the inequality is strict for~$t > n$. 
Moreover, the value~$M_{t}(f_{n,r,\theta}|\D^n)$ is attained with initialization~$x_{t} = -e^{i(\pi-\theta)(n-t)}$,~$0 \le t \le n-1$.
\end{proposition}
\begin{proof}
The general solution to~\eqref{def:recurrence} is~$x_t = r^{t-n} e^{i\theta (t-n)} P(t)$, for degree~$n-1$ polynomial~$P$.
Plugging in~\eqref{def:Lagrange-explicit} the initial condition~$P(t) = L_n[v_{1:n}|t_{1:n}](t)$, with~$t_k = n-k$ and~$v_k = x_{n-k} r^{k} e^{i\theta k}$, for~$k \in [n]$,
\begin{equation}
\label{eq:repeated-roots-implicit}
|x_t| = r^{t-n} |P(t)|
= r^{t-n} \Bigg| \sum_{k \in [n]} x_{n-k} r^{k} e^{i\theta k} \prod_{j \in [n] \setminus k} \frac{t-n+j}{j-k} \Bigg|
\le r^{t-n} \sum_{k \in [n]} r^{k}  \prod_{j \in [n] \setminus k} \frac{t-n+j}{|j-k|}.
\end{equation}
This bound is attained on~$x_{n-k} = e^{-i\theta k} \sign[\prod_{j \in [n]\setminus k} (j-k)] =  e^{-i\theta k} (-1)^{k-1} = -e^{i(\pi-\theta)k}$.
Thus,
\[
M_{t}(f_{n,r,\theta}|\D^n)
= r^{t-n}\sum_{k \in [n]} r^{k} \prod_{j \in [n] \setminus k} \frac{t-n+j}{|j-k|} 
= r^{t-n}\binom{t}{n} \sum_{k \in [n]} \binom{n}{k} \frac{k}{t-n+k} r^{k}
\]
where the last identity follows by simple algebra.
The integral form of the right-hand side follows by
\begin{align*}
\frac{\partial}{\partial r} \sum_{k \in [n]} \binom{n}{k} \frac{k}{t-n+k} r^{t-n+k}  
= n\sum_{k \in [n]} \binom{n-1}{k-1} r^{t-n+k-1}  
= nr^{t-n} (1 + r)^{n-1}.
\end{align*}
Note that for~$t > n$, one has~$\int_{0}^r nu^{t-n}(1+u)^{n-1} du < r^{t-n} \int_{0}^r n(1+u)^{n-1}du = r^{t-n} [(1+r)^{n}-1]$.
\end{proof}
\begin{remark}
{\em The work~\cite{polyak2018peak} linked~$M_t(f_{n,r,\theta}|\D^n)$ to Lagrange interpolation and used~\eqref{eq:repeated-roots-implicit}~\cite[Sec.~2.2]{polyak2018peak}, but recognized neither of the identities in~\eqref{eq:repeated-roots-explicit}, nor the estimate~$M_{t}(f_{n,r,\theta}|\D^n) \le \binom{t}{n}r^{t-n}[(1+r)^n-1]$.
}
\end{remark}
\begin{remark}\vspace{-0.05cm}
{\em 
The integral in~\eqref{eq:repeated-roots-explicit} can be expressed in terms of the hypergeometric function~$\vphantom{F}_2 F_1$~\cite{andrews1999special}.
}
\end{remark}
\subsection{Sharp bound for~$t = n$.}
The special case of Theorem~\ref{th:main-result} with~$t = n$, i.e.~maximization of the {\em instantaneous} amplitude, turns out to be easy. 
For convenience, we provide a standalone proof here.


\begin{proposition}
\label{prop:time-zero}
Theorem~\ref{th:main-result} holds for~$t = n$, so~$\sM_n(\D_n(r_{1:n})|\D_n(w_{1:n})) = M_n(P_n(e^{i\theta}r_{1:n})|\D_n(w_{1:n}))$
for any~$r_{1:n}, w_{1:n} \in \R_+^n$ and~$\theta \in \R$.
In particular, it holds that
$
\sM_n(\D_n(r_{1:n})|\D^n) = \prod_{j \in [n]}(1+r_j) -1.
$
\end{proposition}

\begin{proof}
Any solution to~\eqref{def:recurrence} with~$f \in \cP_n$ satisfies~$x_{n} = -\sum_{k \in [n]} f_{k-1} x_{k-1}$. 
As the result, one has
$
M_{n}(f|\D_n(w_{1:n})) 
= \sum_{k \in [n]} w_k |f_{k-1}|
$, attained on~$x_{k-1} = {w_{k}}{f_{k-1}}/{|f_{k-1}|}$. 
But~$f = P_n[z_{1:n}]$ has coefficients given by~$f_{k-1} = (-1)^{n-k+1}\se_{n-k+1}(z_{1:n})$ in terms of elementary symmetric polynomials:
\[
\se_{d}(x_{1:n}) = \sum_{1 \le i_1 < \dots < i_d \le n} \; \prod_{k \in [d]} x_{i_k}.\vspace{-0.2cm}
\]
As the result,
\begin{equation}
\label{eq:cophase-easy}
\sM_{n}(\D_n(r_{1:n})|\D_n(w_{1:n})) 
= \sup_{z_{1:n} \in \D_n(r_{1:n})} \sum_{k \in [n]} w_k |\se_{n-k+1}(z_{1:n})| 
= \sum_{k \in [n]} w_k \se_{n-k+1}(r_{1:n})
\end{equation}
where the last step is by the triangle inequality. The right-hand side is~$M_{n}(P_n(e^{i\theta}r_{1:n})|\D_n(w_{1:n}))$.
The computation of~$M_n(P_n(e^{i\theta}r_{1:n})|\D^n)$ is similar to the proof of Proposition~\ref{prop:repeated-roots-explicit}, and we omit it.
\end{proof}


\begin{remark}
{\em 
If~$r_k w_k > 0$ for all~$k \in [n]$, the supremum in~\eqref{eq:cophase-easy} is attained {\em only} with~$f = P_n(e^{i\theta} r_{1:n})$, i.e.~with cophase roots of maximal magnitudes.
Indeed,~$|\se_n(z_1, ..., z_n)| \le \se_{n}(|z_1|,\dots,|z_n|) < \se_{n}(r_{1:n})$ if~$\exists k \in [n]$:~$|z_k| < r_k$.
But if~$z_k = e^{i\theta_k} r_k$ for all~$k \in [n]$,~$n \ge 2$, and~$e^{i\theta_{k_1}} \ne e^{i\theta_{k_2}}$ for some~$k_1, k_2 \in [n]$, 
\[
\begin{aligned}
|\se_1(z_{1:n})| 
\le \Big| r_{k_1} + r_{k_2} e^{i\left(\theta_{k_2}-\theta_{k_1}\right)} \Big| + \sum_{k \notin \{k_1,k_2\}} r_k
\le \sqrt{r_{k_1}^2 + r_{k_2}^2 + 2 r_{k_1} r_{k_2} \cos(\theta_{k_2}-\theta_{k_1})} +  \sum_{k \notin \{k_1,k_2\}} r_k.
\end{aligned}
\]
By the Cauchy-Schwarz inequality, the right-hand side is strictly less than~$\sum_{k \in [n]} r_k = \se_1(r_{1:n})$. 
\qed
}
\end{remark}

\begin{remark}
{\em 
Proposition~\ref{prop:time-zero} is reminiscent of Gautschi's bound~\cite{gautschi1962inverses}
for the max-absolute-row-sum norm of the inverse Vandermonde matrix:
$
\| V_n(z_{1:n})^{-1} \|_{\infty,\infty} 
\le \max_{j \in [n]} \prod_{k \ne j} {(1 + |z_j|)}{(|z_j - z_k|)^{-1}}.
$
This bound can be proved with a similar method; doing so without consulting~\cite{gautschi1962inverses} is a good exercise.
}
\end{remark}

\subsection{Crude bound for~$t > n$.}
There is a simple way of extending Proposition~\ref{prop:time-zero} to arbitrary~$t > n$, with a rapid loss of sharpness. While crude, the resulting bounds manage to capture the behavior of~$\sM_t(\D(r)^n|\D^n)$ qualitatively.
The key tool is the recurrence relation for interpolation polynomials~$\psi_t$.

\begin{lemma}
\label{lem:coherency}
Given a simple~$z_{1:n} \in \C^n$ and~$t \in \N_0$, let~$\psi_t(z)$ be the polynomial of degree~$n-1$ that interpolates the monomial~$z^{t}$ on~$z_{1:n}$.
Then~$\psi_{t+1} = A_n(f)^\top \psi_{t}$, where~$A_n(f)$ is the companion matrix, cf.~\eqref{def:companion-matrix}, of the characteristic polynomial~$f = P_n[z_{1:n}]$ in~\eqref{def:recurrence}. 
Equivalently, for~$t \in \N_0$ it holds that
\begin{align}
\label{eq:poly-rec}
\psi_{t+1}(z) &= z\psi_{t}(z) - \frac{\psi_{t}^{(n-1)}(0)}{(n-1)!} f(z).
\end{align}
\end{lemma}

\begin{proof}
It suffices to verify~\eqref{eq:poly-rec}, which is clearly equivalent to~$\psi_{t+1} = A_n(f){}^{\top} \psi_t$; moreover, the case~$t \le n-1$ is obvious.
Now for~$t \ge n$, assume that~$\psi_{t-1} = z^{t-1}$ on~$z_{1:n}$, and~$\deg(\psi_{t-1}) \le n-1$. 
Clearly,~$z\psi_{t-1}(z) = z^{t}$ on~$z_{1:n}$, and has degree at most~$n$. Since~$f(z) = 0$ on~$z_{1:n}$, and~$f \in \cP_n$, polynomial~$\psi_{t}$ in right-hand side of~\eqref{eq:poly-rec} interpolates~$z^{t}g(z)$ on~$z_{1:n}$, and one has~$\deg(\psi_t)\le n-1$. 
\end{proof}

\begin{remark}
{\em In the natural sense, the vector equation~$\psi_{t+1} = A_n(f){}^{\top} \psi_t$ is the adjoint of~\eqref{def:recurrence-vectorized}.
Replacing~$\psi_t(z)$ with polynomials interpolating~$z^t g(z)$ for~$g \in \cP_{n}$, with Lemma~\ref{lem:coherency} preserved, we get the~$n$-dimensional subspace of~$\C^{\N_0 \times [n]}$, adjoint to that of possible trajectories of~\eqref{def:recurrence-vectorized} in~$\C^{\N_0 \times [n]}$.}
\end{remark}

\noindent
By Fact~\ref{fact:interpolation}, bounding~$\sM_{t}(\D(r)^n|\D^n)$ amounts to bounding the sum of absolute coefficients~$\|\psi_{t}\|_1$ for the Lagrange polynomial~$\psi_{t} = L_n[(z_{1:n})^t|z_{1:n}]$ on~$z_{1:n}\in \D(r)^n$, and this can be done  via Lemma~\ref{lem:coherency}.
\begin{proposition}
\label{prop:simple-bound}
For all~$t \in \N_n$ and~$r > 0$, it holds that~$\sM_{t}(\D(r)^n|\D^n) \le (\frac{3n-1}{2}r)^{t-n}\, [(1+r)^n-1]$.
\end{proposition}
\begin{proof}
Index~$k$ will correspond to~$n-k$ in the proof of Proposition~\ref{prop:repeated-roots-explicit}; this allows to ease the notation in both cases.
Proceeding by induction, we prove that, for the individual coefficients of~$\psi_{t}$,
\begin{equation}
\label{eq:simple-bound-individual-coeffs}
\frac{|\psi_{t}^{(k)}(0)|}{k!} \le \left(\frac{3n-1}{2}\right)^{t-n} \binom{n}{k} r^{t-k}, \quad 0 \le k \le n-1.
\end{equation}
The case~$t = n$ can be recovered from the proof of Proposition~\ref{prop:time-zero}. 
Next, by invoking Lemma~\ref{lem:coherency},
\begin{equation}
\label{eq:simple-bound-triangle}
\frac{|\psi_{t+1}^{(k)}(0)|}{k!} 
\le 
\left| \frac{\psi_{t}^{(n-1)}(0)}{(n-1)!} \cdot \frac{f^{(k)}_{\vphantom{t}}(0)}{k!} - \frac{\psi_{t}^{(k-1)}(0)}{(k-1)!} \right|
\le \frac{|\psi_{t}^{(n-1)}(0)|}{(n-1)!} \se_{n-k}(r\ones_{n}) + \frac{|\psi_{t}^{(k-1)}(0)|}{(k-1)!}, \quad k > 0;
\end{equation}
ditto without the second term in the case of~$k = 0$.
Whence, by plugging this in~\eqref{eq:simple-bound-individual-coeffs}, we arrive at
\[
\begin{aligned}
\frac{|\psi_{t+1}^{(k)}(0)|}{k!} 
\le \left( \frac{3n-1}{2} \right)^{t-n} F_{n,k}\; r^{t-k+1} \quad\text{with}\quad F_{n,k} := n \binom{n}{k} + \binom{n}{k-1}.
\end{aligned}
\]
To complete the induction step, it remains to observe that~$F_{n,k} =  \big( n + \frac{k}{n-k+1}\big) \binom{n}{k} \le \frac{3n-1}{2}\binom{n}{k}$.
\end{proof}

\subsection{Discussion}
The bound obtained in Proposition~\ref{prop:simple-bound} is dramatically worse than the estimate
\begin{equation}
\label{eq:vanilla-bound}
\sM_t(\D(r)^n|\D^n) \le \binom{t}{n} r^{t-n} [(1+r)^n-1],
\end{equation}
following from Proposition~\ref{prop:repeated-roots-explicit} and Theorem~\ref{th:main-result}. 
In particular, in the asymptotic regime~$t \to + \infty$ with~$n/t \to 0$, the binary logarithm of the ratio of these two estimates grows as~$(1+o(1)) \, t \log_2(\frac{3n-1}{4})$.
In fact, the Cauchy bound (Proposition~\ref{prop:monomial-interpolation-0}), after Fourier analysis, gives a better estimate for~$r = 1$,
\[
\sM_t(\D^n|\D^n) \le \sqrt{n} \binom{t}{n}2^n + \sqrt{n},
\] 
see~Corollary~\ref{cor:Fourier-bound}; however, it seems nontrivial to generalize it for~$r < 1$ in a clean way.
Moreover, in view of~\cite[Proposition~2.1]{ostrovskii2024near}, it is conceivable that the extraneous~$\sqrt{n}$ factor in the latter bound can be replaced with a constant~$c > 1$ via purely analytical techniques, avoiding the tools we used to prove Theorem~\ref{th:main-result}. 
In any case, improvements of such sort are outside of the scope of this work.

Meanwhile, replacing the factor~$r^{t-n}[(1+r)^n-1]$ in~\eqref{eq:vanilla-bound} with the correct one~$\int_{0}^r n u^{t-n}(1 + u)^{n-1} du$, and thus proving Theorem~\ref{th:main-result}, is a delicate task if viewed through the analysis lens.
Indeed, the proof of Proposition~\ref{prop:simple-bound} suggests that the looseness of~\eqref{eq:simple-bound-individual-coeffs} is due to the triangle inequality we employed in~\eqref{eq:simple-bound-triangle}. 
If we instead {\em speculate} that the two terms in~\eqref{eq:simple-bound-triangle} have {\em opposite} signs, as suggested by~\eqref{eq:poly-rec}, the corresponding adjustment of the subsequent calculation in~\eqref{eq:simple-bound-triangle} would result in the correct bound\vspace{-0.1cm}
\begin{equation}
\label{eq:correct-bound-individual-coeffs}
\frac{|\psi_{t}^{(k)}(0)|}{k!} \le \binom{t}{n} \binom{n}{k} \frac{n-k}{t-k}r^{t-k}, \quad 0 \le k \le n-1
\end{equation}
for the individual coefficients of~$\psi_{t}(z)$. Indeed, the corresponding modification of~\eqref{eq:simple-bound-triangle} implies that
\[
\begin{aligned}
\frac{|\psi_{t+1}^{(k)}(0)|}{k!} 
&\le \binom{t}{n} F_{t,n,k}\; r^{t-k+1}  \quad \text{with} \quad F_{t,n,k} := \binom{n}{k} \frac{n}{t-n+1} - \binom{n}{k-1} \frac{n-k+1}{t-k+1}.
\end{aligned}
\]
By simple algebra, we get~$\binom{t}{n} F_{t,n,k} = \binom{t}{n} \binom{n}{k} ( \frac{n}{t-n+1} - \frac{k}{t-k+1} ) = \binom{t+1}{n} \binom{n}{k} \frac{n-k}{t-k+1}$, so~\eqref{eq:correct-bound-individual-coeffs} holds for~$t + 1$.
This calculation elucidates that in order to prove Theorem~\ref{th:main-result}, it is crucial to control the relative phases of the terms that comprise the interpolation polynomial in~\eqref{eq:poly-rec}. 
The latter task does not trivialize even if we restrict~$z_{1:n}$ to be real, to deal with signs instead of phases: as we discuss in Section~\ref{sec:shadrin}, showing such sign alternation properties for real interpolation polynomials is known to be challenging~\cite{shadrin1995error}.
Establishing the relation of this challenge to Schur positivity is our key insight.

Next, we give a concise review of symmetric functions and partitions, needed to prove Theorem~\ref{th:main-result}.

\section{Background on symmetric functions}
\label{sec:partitions}
This section covers our subsequent needs. 
More can be found in~\cite[Chap.~7]{stanley2023volume2}, \cite{macdonald1998symmetric},~\cite{tamvakis2012theory}, and~\cite{prasad2018introduction}.\vspace{-0.1cm}
%
%
\subsubsection*{Partitions}
Any nonincreasing sequence~$\lambda = (\lambda_1, \lambda_2, \dots)$ with~$\lambda_k \in \Z_+$ is called {\em a partition}, and~$\Par$ denotes the set of all partitions.
The nonzero entries in~$\lambda \in \Par$ are {\em parts}; its {\em length}~$\ell(\lambda)$ is the number of parts ($\lambda_i = 0$ is often omitted).
Its {\em size}~$|\lambda|$ is the sum of parts;~$\lambda \vdash d$ is the same as~$|\lambda| = d$.
For~$n, d \in \Z_+$, we define the sets~$\Par(d) := \{\lambda \in \Par: |\lambda| = d\}$ and~$\Par_n(d) = \{\lambda \in \Par(d): \ell(\lambda) = n\}$. 
The {\em Young diagram} of~$\lambda$ is the left-justified two-dimensional array of cells with~$\lambda_i$ cells in each row~$i $.
For~$\lambda,\mu \in \Par$, we write~$\lambda \contains \mu$ if~$\lambda - \mu$ is a nonnegative vector, i.e.~the Young diagram of~$\lambda$ contains that of~$\mu$.
The {\em conjugate} partition of~$\lambda$, denoted with~$\lambda'$, is obtained by transposing the Young diagram of~$\lambda$.
A {\em hook} partition has the form~$(a+1,1^b)$ with~$a,b \ge 0$, where~$1^b$ is~$1$ repeated~$b$ times; its Young diagram is composed of a single row and a single column.
In the {Frobenius notation}, this partition writes as~$(a|b)$, and its conjugate is~$(b|a)$.
Any~$\lambda \in \Par$ can be decomposed into hooks by iteratively removing the boundary hook; 
this allows to write any finite partition as~$(a_{1:s}\,|\,b_{1:s})$, where~$(a_1|b_1),\dots, (a_s|b_s)$ is the sequence of hooks removed. 
For example,~$(4,3,3,1) = (3,1,0\,|\,3,1,0)$.\vspace{-0.1cm}


\subsubsection*{Symmetric functions}
\label{sec:symmetric-functions}
We let~$\sLambda$ be the algebra of symmetric functions in countably many indeterminates~$x_1, x_2, \dots$  with rational coefficients; its elements are formal power series in~$x_1, x_2, \dots$ with rational coefficients, invariant under any permutation of the variables.
Special elements of~$\sLambda$ are {elementary} symmetric functions~$\se_d$ and~{complete homogeneous} symmetric functions~$\sh_d$, as given by
\begin{equation}
\label{def:elem-comp}
\begin{aligned}
\se_d := \sum_{i_1 < \,\dots\, < i_d} \prod_{k \in [d]} x_{i_k},
\quad\quad
\sh_d := \sum_{i_1 \le \,\dots\, \le i_d} \prod_{k \in [d]} x_{i_k}
\end{aligned}
\end{equation}
for~$d \in \N$, with~$\se_0 = \sh_0 = 1$ and~$\se_d = \sh_d = 0$ for~$d < 0$.
We let~$\sLambda^d$ be the set of homogeneous symmetric functions {\em of degree~$d$} in~$\sLambda$, and we let~$\sLambda^{\le d} := \bigoplus_{k = 0}^d \sLambda^k$; in particular,~$\se_d, \sh_d \in \sLambda^d$ and~$\sum_{k \le d} \se_k \in \sLambda^{\le d}$.
For~$\sg \in \sLambda^d$ and~$z_{1:n} \in \C^n$, we denote with~$\sg(z_1,\dots,z_n)$ the {\em specialization} of~$\sg$ to the variables~$x_k = z_k \ones_{k \le n}$,~$k \in \N$; thus,~$\sg(z_1, \dots, z_n)$ is a symmetric polynomial of degree~$d$ on~$\C^n$. 
%
%

%

\subsubsection*{Monomial basis and~$\sm$-positivity}
Any~$\lambda \in \Par$ specifies the {\em symmetric monomial}
$
\sm_{\lambda} \in \sLambda,
$
obtained by applying all possible permutations~$(x_1, x_2, \dots) \mapsto (x_{\sigma(1)}, x_{\sigma(2)}, \dots)$ of variables to the monomial~$\prod_{i \in \N} x_{i}{}^{\lambda_i}$ and summing the resulting {\em distinct} monomials. 
The set of all symmetric monomials, with the partial ordering inherited from~$\Par$, is a basis of~$\sLambda$ as a vector space (over~$\Q$), and symmetric monomials of degree~$d$ similarly form a basis for~$\sLambda^d$. 
In other words, any~$\sg \in \sLambda$ can be uniquely expressed as~$\sum_{\lambda \in \Par} c_{\lambda} \sm_{\lambda}$ with~$c_\lambda \in \Q$, and with~$c_{\lambda} \ne 0$ only for~$\lambda \in \Par(d)$ when~$\sg \in \sLambda^d$. 
Furthermore,~$\sg \in \sLambda$ is called~{\em $\sm$-positive} if its expansion in the~$\sm$-basis has~$c_{\lambda} \ge 0$ for all partitions.
The importance of this concept for our purposes is suggested by the following immediate observation.
\begin{fact}
\label{fact:monomial-maximization}
If~$\sg \in \sLambda^{\le d}$ is~$\sm$-positive, the maximum of~$|\sg(z_{1:n})|$ over~$\D_n(r_{1:n})$ is attained at~$z_{1:n} = r_{1:n}$. 
If one has~$\sg \in \sLambda^{d}$ as well, then the same maximum is attained at~$z_{1:n} = e^{i\theta} r_{1:n}$ with arbitrary~$\theta \in \R$. 
\end{fact}
%

Starting with~\eqref{def:elem-comp}, we construct symmetric functions~$\se_{\lambda} = \prod_{i \in \N} \se_{\lambda_i}$ and~$\sh_{\lambda} = \prod_{i \in \N} \sh_{\lambda_i}$. 
These turn out to give two other bases of~$\sLambda$---{\em$\se$-basis} and~{\em$\sh$-basis}---that are mutually conjugate, i.e.~$\se_{\lambda} = \sh_{\lambda'}$.
Clearly,~$\se_{\lambda}$ and~$\sh_{\lambda}$ are~$\sm$-positive, so the positivity in either of the two bases implies~$\sm$-positivity. 
In general, for two bases~$\mathsf{u},\mathsf{v}$ of~$\sLambda$, we say that~$\mathsf{v}$ is~{\em$\mathsf{u}$-positive} if each~$\mathsf{v}_{\lambda} \in \mathsf{v}$ expands positively in~$\mathsf{u}$.
To show that~$\sg \in \sLambda$ is~$\sm$-positive, it suffices to show its positivity in some auxiliary~$\sm$-positive basis.



\subsubsection*{Schur functions and~$\sch$-positivity.}
One special basis of~$\sLambda$ is comprised of the {\em Schur functions}~$\sch_{\lambda}$. 
These functions, first studied by Cauchy~\cite{cauchy1815memoire}, admit several equivalent definitions, but none of these is as tangible as for~$\sm_{\lambda}$,~$\se_{\lambda}$ or~$\sh_{\lambda}$; 
we choose the definition that immediately shows their~$\sm$-positivity.

A {\em semi-standard Young tableau} (SSYT) with {\em shape}~$\lambda \in \Par$ is a two-dimensional array~$T$ that fills the cells of the Young diagram of~$\lambda$ with positive integers, such that the entries (a) srtictly increase in each column; (b) do not decrease in each row. We say that~$T$ has {\em type}~$\alpha$, where~$\alpha = (\alpha_1, \alpha_2, \dots)$ with~$\alpha_k \in \Z_+$ for all~$k \in \N$, if~$\alpha_k$ is the the number of occurences of~$k$ in~$T$. 
{\em Schur function~$\sch_{\lambda}$} reads
\begin{equation}
\label{def:schur-via-monomials}
\sch_{\lambda} = \sum_{\mu \in \Par} K_{\lambda \mu} \sm_{\mu}
\end{equation}
where the {\em Kostka number}~$K_{\lambda\mu}$ is the number of SSYTs of shape~$\lambda$ and type~$\mu$, and~$\mu \notin \Par(|\lambda|)$ can be dropped.
From this definition, we see that~$\sch_{\lambda}$ is~$\sm$-positive. 
In fact, there is an indexing of Kostka numbers over~$\Par \times \Par$, such that the infinite matrix~$(K_{\lambda\mu})_{\lambda, \mu \in \Par}$ is lower-unitriangular,
and is the Cholestky factor of the transition matrix from~$\sm$-basis to~$\sh$-basis (\cite[Cor.~7.12.3]{stanley2023volume2}); thus,
\begin{equation*}
\sh_{\mu} = \sum_{\lambda \in \Par} K_{\lambda \mu} \sch_{\lambda}. 
\end{equation*}
Hence,~$\sch$-positivity is weaker than~$\sh$-positivity (or~$\se$-positivity, by conjugating~$\mu$ in the last equation) but still implies~$\sm$-positivity, so the latter can be established by studying~$\sg \in \sLambda$ in the Schur basis.\footnote{This methodology bears a striking similarity to the Burer-Monteiro relaxation in semidefinite programming~\cite{burer2003nonlinear,burer2005local}. 
This is hardly a conincidence; we believe that investigating the tentative connection is a worthwhile research direction.}

\subsubsection*{Schur functions for hook partitions.}
From~\eqref{def:schur-via-monomials}, we obtain~$\mathsf{s}_{(a)} = \sh_{a}$ and~$\mathsf{s}_{(1^b)} = \se_{b}$. 
More generally, the {\em Pieri rule} (see e.g.~\cite[Thm.~7.5.17]{stanley2023volume2}) gives the Schur expansion of the product
$
\sh_{a} \sch_\lambda  = \sum_{\mu \,\in\, \textsf{Add}_{a}(\lambda)} \sch_\mu
$
where~$\textsf{Add}_{a}(\lambda)$ is the set of partitions obtained by adding~$a$ cells to the Young diagram of~$\lambda$, {\em no two in the same column}.
For~$\lambda = (1^b)$, this set consists of only two hook partitions, and we conclude that
\begin{equation}
\label{eq:hook-derivative}
\sh_{a} \se_{b} = \sch_{(a|b-1)} + \sch_{(a-1|b)}.
\end{equation}
A telescoping argument then gives explicit, bilinear in~$\sh,\se$ formulae for hook-shaped Schur functions:
\begin{fact}
\label{fact:hook-bilinear}
For any~$a,b \in \Z_+$, one has~$\sch_{(a|b)} = \sum_{j = 0}^a (-1)^{j} \sh_{a-j} \se_{b+j+1} = \sum_{k = 0}^b (-1)^{k} \sh_{a+k+1} \se_{b-k}$. 
Note that in both expressions, we might w.l.o.g.~replace the range of summation with the entire~$\Z_+$.
\end{fact}

%
\subsubsection*{Hook-content formula}
\label{sec:hook-content-formula}
For~$\lambda \in \Par$ with~$|\lambda| \le n$, the quantity~$\sch_{\lambda}(\ones_n) = \sch_{\lambda}(x_{1:n})|_{x_{1:n} = (1,\dots, 1)}$ is the number of SSYTs with shape~$\lambda$ and entries at most~$n$. 
In fact, it admits the explicit expression
\begin{equation}
\label{sec:hook-content-formula-general}
\sch_{\lambda}(\ones_n) = \prod_{j \,\in\, [\ell(\lambda)]} \, \prod_{k \,\in\, [\lambda_j]}\frac{n + k - j}{w_{\lambda}(j,k)+1} \quad \text{with}\quad w_{\lambda}(j,k) := \lambda_j^{\vphantom'} - j + \lambda_{k}' - k,
\end{equation}
known as the {\em hook-content formula}~(e.g.~\cite[Cor.~7.21.4]{stanley2023volume2}).
Here, each factor corresponds to a specific cell~$(j,k)$ of the Young diagram;~$w_{\lambda}(j,k)$ is the number of cells to the right or below~$(j,k)$. 
In the case of hook partitions, after some simple manipulations with binomial coefficients we conclude that
\begin{equation}
\label{eq:hook-content-formula}
\sch_{(a|b)}(\ones_n) 
= \frac{n}{a+b+1} \Bigg( \frac{1}{a!} \prod_{j \in [a]} (n+j) \Bigg) \Bigg( \frac{1}{b!} \prod_{k \in [b]} (n-k) \Bigg)
= \binom{n+a}{a+b+1} \binom{a+b}{b}.
\end{equation}
In particular, we have invariance under transposition of the hook, namely~$\sch_{(a|b)}(\ones_{N-a}) = \sch_{(b|a)}(\ones_{N-b})$.
\section{Proof of Theorem~\ref{th:main-result}}
\label{sec:result}

The first step of the proof is an estimate for the error of Lagrange interpolation of a monomial; 
for convenience, we present it as a separate proposition. This estimate relies on the classical observation:
\begin{lemma}[{\cite[III.8.4]{freud1971book}}]
\label{lem:integral-representation}
Let~$g$ be holomorphic on~$\D(r)$,~$z_{1:n}$  simple,~$\|z_{1:n}\|_{\infty} < r$, and~$f = P_n[z_{1:n}]$.
For~$z \in \C: |z| < r$, the Lagrange polynomial~$\hat g = L_n[g_{1:n}|z_{1:n}]$ with~$g_{1:n} := (g(z_1), \dots, g(z_n))$ satisfies
\begin{equation}
\label{eq:integral-representation}
g(z) - \hat g(z) = \frac{f(z)}{2\pi i} \oint_{\T} \frac{g(\xi)}{f(\xi) (\xi - z)} d \xi.
\end{equation}
\end{lemma}
\begin{proof}
Define~$h_z(\zeta) := f(\zeta) (\zeta-z)$. For~$z \notin \{z_1, \dots, z_n\}$, dividing the left-hand side of~\eqref{eq:integral-representation} over~$f(z)$,
\[
\frac{g(z) - \hat g(z)}{f(z)} 
\overset{\eqref{def:Lagrange-explicit}}{=} \lim_{\zeta \to z} \frac{g(\zeta)}{h_z(\zeta)} (\zeta - z) + \sum_{k \in [n]} \lim_{\zeta \to z_k} \frac{g(\zeta)}{h_z(\zeta)} (\zeta - z_k)
= \frac{1}{2\pi i} \oint_{\T} \frac{g(\xi)}{f(\xi) (\xi - z)} d \xi
\]
due to the residue theorem.
Similarly, the right-hand side of~\eqref{eq:integral-representation} vanishes at~$z \in \{z_1, \dots, z_n\}$.
\end{proof}

We remind that~$\se_d,\sh_d$ are, respectively, the elementary and complete homogeneous symmetric {\em functions} of degree~$d$, cf.~\eqref{def:elem-comp};
their specializations to ~$n$ variables~$z_{1:n}$ are denoted with~$\se_d(z_{1:n})$,~$\sh_{d}(z_{1:n})$.

\begin{proposition}
\label{prop:monomial-interpolation-0}
Given any~$z_{1:n} \in \D_n(r_{1:n})$, the Lagrange polynomial~$\psi_{t} = L_n[(z_{1:n})^t|z_{1:n}]$ for the monomial~$z^t$,~$t \in \N_n$, admits the pointwise error bound~(valid for all~$z \in \C$) in terms of~$f = P_n[z_{1:n}]$:
\begin{equation*}
|\psi_{t}(z) - z^t|
\le |f(z)| \sum_{d = 0}^{t-n} \sh_{d}(r_{1:n}) \, |z|^{t-n-d},
\end{equation*}
where the equality is attained with~$z_{1:n} = e^{i\theta}r_{1:n}$ for any~$\theta \in \R$.
In particular, in the case~$r_{1:n} = r\ones_n$,
\begin{equation}
\label{eq:monomial-interpolation-0-uniform}
| \psi_{t}(z) - z^t | \le |f(z)| \sum_{d = 0}^{t-n} {n+d-1 \choose d}  r^{d} |z|^{t-n-d}.
\end{equation}

\end{proposition}

\begin{proof}
Let~$\veps_t(z) := z^t - \psi_t(z)$. 
We can w.l.o.g.~assume that~$t \in \N_{n+1}$ and~$z \notin \{z_1, \dots, z_n\}$, the complementary cases being trivial. 
Also, we assume that~$z_{1:n}$ is simple; the general case can then be verified via continuity. 
Invoking Lemma~\ref{lem:integral-representation} with~$g(z) = \psi_{t}(z)$~and fixing~$r > \max\{|z|,\|r_{1:n}\|_{\infty}\}$,
%
%
\begin{align}
\label{eq:monomial-interpolation-0-integral}
\veps_t(z) 
= \frac{f(z)}{2\pi i} \oint_{\T(r)} \frac{\xi^t}{(\xi - z) f(\xi)} d\xi 
= \frac{f(z)}{2\pi i} \oint_{\T(r)} \frac{\xi^{t-n-1}}{(1 - z\xi^{-1}) \prod_{k \in [n]} (1 - z_k\xi^{-1})} d\xi.
\end{align}
As the result, introducing~$z_0 = z$ for convenience, we get
\begin{align}
\veps_t(z_0) 
= \frac{f(z_0)}{2\pi i} \oint_{\T(r)} \Bigg( \prod_{k = 0}^n  \sum_{s_k \in \N_0} z_k^{s_k} \xi^{-s_k} \Bigg)  \xi^{t-n-1} d\xi 
&= \frac{f(z_0)}{2\pi i} \oint_{\T(r)} \sum_{d \in \N_0} \sh_{d}(z_{0:n}) \xi^{t-n-d-1} d\xi \notag\\
&=f(z_0) \sh_{t-n}(z_{0:n}).
\label{eq:complete-extraction}
\end{align}
Here, we used the geometric progression identity (clearly,~$|z_k\xi^{-1}| < 1$ for~$\xi \in \T(r)$ and~$0 \le k \le n$), collected similar terms, and invoked the residue theorem.
Since~$\sh_{d} \in \sLambda^d$ is~$\sm$-positive, Fact~\ref{fact:monomial-maximization} gives 
\[
\frac{|\veps_t(z)|}{|f(z)|} 
\le \sh_{t-n}(r_{1:n},|z|)  
= \sum_{d=0}^{t-n} \sh_{d}(r_{1:n}) |z|^{t-n-d}
\]
and the first claim follows. Finally, we get~\eqref{eq:monomial-interpolation-0-uniform} since,  by simple combinatorics,~$\sh_{d}(\ones_n) = \binom{n+d-1}{d}$.
\end{proof}

Before we return to the proof, let us make one observation relevant to the discussion in Section~\ref{sec:warmup}.
\begin{corollary}
\label{cor:Fourier-bound} 
For any~$z_{1:n} \in \D^n$ and~$t \in \N_n$, it holds that~$\| \psi_{t}\|_1 \le \sqrt{n}\| \psi_{t}\|_2 \le \sqrt{n} [2^n \binom{t}{n} + 1]$.
\end{corollary}
\begin{proof}
The first inequality holds trivially, since~$\psi_{t} \in \C^{n}$. 
For the second one, we first observe that 
\[
\begin{aligned}
\|\psi_{t}\|_2
= \left( \frac{1}{2\pi}\int_{0}^{2\pi} |\psi_{t}(e^{i\theta})|^2 d\theta \right)^{1/2} 
\le \sup_{|z| = 1} |\psi_{t}(z)|
\overset{\eqref{eq:monomial-interpolation-0-uniform}}{\le}
	  \binom{t}{n} \sup_{|z| = 1} |f(z)| + 1
\le  \binom{t}{n} 2^n + 1.
\end{aligned}
\]
Here, we first used Parseval's identity, then H\"older's inequality, then the hockey-stick identity for the binomial coefficients, and finally~maximized~$|f(z)|=\prod_{j \in [n]}^n|z-z_j|$ over~$z,z_1, \dots, z_n \in \T$.
\end{proof}

The error bound of Proposition~\ref{prop:monomial-interpolation-0} will be used in Section~\ref{sec:shadrin}. 
Our present goal is to complete the proof of Theorem~\ref{th:main-result}, to which end we express the coefficients of~$\psi_{t}(z)$ as symmetric polynomials.

\begin{proposition}
\label{prop:monomial-interpolation-any-k}
Under the premise of Proposition~\ref{prop:monomial-interpolation-0}, for~$t \in \N_n$ and~$0 \le k  \le n-1$, it holds that
\begin{align*}
-\frac{\psi_{t}^{(k)}(0) }{k!}
\;=\; \sum_{j = 0}^{\min\{k,\,t-n\}} (-1)^{n-k+j} \, \se_{n-k+j}(z_{1:n}) \, \sh_{t-n-j}(z_{1:n})
\;=\; (-1)^{n-k} \sch_{(t-n\,|\,n-k-1)}(z_{1:n})
\end{align*}
where, as defined in~Section~\ref{sec:partitions},~$\sch_{(a|b)}$ is the Schur function of the hook partition~$(a|b) = (a+1,1^{b})$.
\end{proposition}

\begin{proof}
Sequential differentiation of~\eqref{eq:monomial-interpolation-0-integral} via the Leibniz formula gives the following for~$r > \|r_{1:n}\|_{\infty}$,
\begin{equation*}
\binom{t}{k} z^{t-k}-\frac{\psi_t^{(k)}(z)}{k!} 
= \sum_{j = 0}^{k} \binom{k}{j}  \frac{j!}{k!}  \frac{f^{(k-j)}(z)}{2\pi i}\oint_{\T(r)} \frac{\xi^{t}}{f(\xi) (\xi-z)^{j+1}} d \xi.
\end{equation*}
(As in the proof of Proposition~\ref{prop:monomial-interpolation-0}, we can assume w.l.o.g. that~$z_{1:n}$ is simple.)
Plugging in~$z = 0$,
\begin{align*}
-\frac{\psi_t^{(k)}(0)}{k!} 
= \sum_{j = 0}^{k} \frac{f^{(k-j)}(0)}{(k-j)!}  \frac{1}{2\pi i}\oint_{\T(r)} \frac{\xi^{t-j-1}}{f(\xi)} d \xi 
= \sum_{j = 0}^{k} (-1)^{n-k+j} \se_{n-k+j}(z_{1:n}) \frac{1}{2\pi i} \oint_{\T(r)} \frac{\xi^{t-j-1}}{f(\xi)} d \xi.
\end{align*}
Furthermore, proceeding as in~\eqref{eq:complete-extraction} we compute
\[
\begin{aligned}
\frac{1}{2\pi i} \oint_{\T(r)} \frac{\xi^{t-j-1}}{f(\xi)} d \xi
= \frac{1}{2\pi i} \oint_{\T(r)} \frac{\xi^{t-n-j-1}}{\prod_{\ell \in [n]} (1 - z_\ell\xi^{-1})} d\xi
&= \frac{1}{2\pi i} \oint_{\T(r)} \Bigg( \prod_{\ell \in [n]} \sum_{s_\ell \in \N_0} z_\ell^{s_\ell} \xi^{-s_\ell} \Bigg) \xi^{t-n-j-1} d\xi \\
&=\sh_{t-n-j}(z_{1:n}).
\end{aligned}
\]
(Due to the convention that~$\sh_d = 0$ for~$d < 0$, this is valid for any~$j \in \N_0$.)
The first claimed identity follows since~$\se_{n-k+j}(z_{1:n}) \sh_{t-n-j}(z_{1:n}) = 0$ for~$j > \min\{k,t-n\}$; 
the second one from Fact~\ref{fact:hook-bilinear}.
\end{proof}

We complete the proof of Theorem~\ref{th:main-result}. 
Invoking~Fact~\ref{fact:interpolation}, then Proposition~\ref{prop:monomial-interpolation-any-k}, and finally Fact~\ref{fact:monomial-maximization}, 
\begin{align}
\sM_{t}(\D_n(r_{1:n})|\D_n(w_{1:n}))
= \sum_{k \in [n]} w_k \frac{\big|\psi_{t}^{(k-1)}(0)\big|}{(k-1)!} 
&= \sup_{z_{1:n} \,\in\, \D_n(r_{1:n})}\sum_{k \in [n]} w_k |\sch_{(t-n|n-k)}(z_{1:n})| \notag\\
&= \sum_{k \in [n]} w_k \sch_{(t-n|n-k)}(r_{1:n})
= M_{t}(P_n[e^{i\theta} r_{1:n}] | \D_n(w_{1:n})).
\label{eq:point-amp-value}
\end{align}
Note that the last equation is delivered with the initial data~$x_{k-1} = -w_{k} e^{i(\pi-\theta)(n-k+1)}$,~$k \in [n]$.\qed

\begin{remark}
{\em 
Equation~\eqref{eq:point-amp-value} expresses the value~$\sM_t(\D_n(r_{1:n})|\D_n(w_{1:n}))$ in terms of Schur polynomials, which allows to compute it for any~$(r_{1:n}, w_{1:n})$.
However, a more practical way to compute this value is by using the bilinar representation of~$\sch_{(a|b)}$ in terms of~$\sh_d,\se_d$ (cf.~Fact~\ref{fact:hook-bilinear} and~Proposition~\ref{prop:monomial-interpolation-any-k}) in combination with Newton's identities, using which we can compute~$\sh_d(r_{1:n}),\se_{d}(r_{1:n})$ iteratively.
}
\end{remark}

\subsection{Deriving~$\sM_t(\D(r)^n|\D^n)$ from the hook-content formula.}
With~\eqref{eq:point-amp-value} at hand, we can compute~$\sM_t(\D(r)^n|\D^n)$ by using the special case~\eqref{eq:hook-content-formula} of the hook-content formula (see~Section~\ref{sec:hook-content-formula}):
\[
\begin{aligned}
\sum_{k \in [n]} \sch_{(t-n|k-1)}(r\ones_n) 
&=
	\sum_{k \in [n]} \binom{t}{n-k} \binom{t-n+k-1}{k-1}  r^{t-n+k} \\
&= \sum_{k \in [n]} \binom{t}{n-k} \binom{t-n+k}{k} \frac{k r^{t-n+k}}{t-n+k}   
=  \binom{t}{n} \sum_{k \in [n]} \binom{n}{k}   \frac{kr^{t-n+k}}{t-n+k} \;,
\end{aligned}
\]
where the last step is by simple algebra.
As expected, this is the same expression as in Proposition~\ref{prop:repeated-roots-explicit}. 

\begin{remark}
{\em The above argument suggests an alternative---and likely unknown---proof of the hook-content formula~\eqref{sec:hook-content-formula-general}. 
Indeed, the specialization~\eqref{eq:hook-content-formula} of the formula for hook partitions follows by adapting the calculation in the proof of Proposition~\ref{prop:repeated-roots-explicit} to~$w_{1:n} = (1,0,\dots,0)$.
Then, the Giambelli identity
$
\sch_{(a_{1:s}|b_{1:s})} = \det(\sch_{a_j|b_k})_{j,k \,\in [s]},
$
likely allows to derive~\eqref{sec:hook-content-formula-general} from~\eqref{eq:hook-content-formula} combinatorially; cf.~\cite{eugeciouglu1988combinatorial}.
}
\end{remark}
\section{Reinhardt domains}
\label{sec:reinhardt}

Let us discuss the implications of Theorem~\ref{th:main-result} in the scenario where~$\RR,\SS$ are Reinhardt domains.
\begin{definition}
\label{def:reinhardt-domain}
$\boldsymbol{D} \subseteq \C^n$ is a {\em Reinhardt domain centered at~$a_{1:n}$} if together with any~$z_{1:n} \in \C^n$,~$\boldsymbol{D}$ contains any~$\wt z_{1:n} \in \C^n: |\wt z_k - a_k| = |z_k - a_k|$ for~$k \in [n]$. 
If~``$=$'' is replaced with~``$\le$'', $\boldsymbol{D}$ is {\em complete}.
\end{definition}
Equivalently, Reinhard domain is simply a union of tori, or of polydiscs in the complete case.
Our subsequent discussion is limited to {\em origin-centered} Reinhardt domains, which contain any~$z_{1:n}$ together with the torus~$\T_n(r_{1:n})$ where~$r_k = |z_k|$, or the polydisc~$\D_n(r_{1:n})$ in the complete case.
Let
\[
\begin{aligned}
\boldsymbol{D}_+ &:= \{(|z_{1}|, \dots, |z_{n}|): z_{1:n} \in \boldsymbol{D}\},\\
\boldsymbol{D}_{\pm} &:= \{(\pm|z_{1}|, \dots, \pm|z_{n}|): z_{1:n} \in \boldsymbol{D}\}.
\end{aligned}
\]

\begin{corollary}
\label{cor:Reinhardt-to-orthant}
Let~$\RR,\SS$ be two compact and origin-centered Reinhard domains, then for~$t \in \N_n$,
\[
\sM_t(\RR|\SS) = \sM_t(\RR_+|\SS_{\pm}) = \sup_{r_{1:n} \in \RR_+,\; w_{1:n} \in \SS_+} \sum_{k \in [n]}  w_k \sch_{(t-n|n-k)}(r_{1:n}). 
\]
\end{corollary}
\begin{proof}
It is clear that~$\sM_t(\RR|\SS) \ge \sM_t(\RR_+|\SS_{\pm})$, cf.~\eqref{def:instant-amplitude}--\eqref{def:max-instant-amplitude}.
For the reverse inequality, observe that
\[
\sM_t(\RR|\SS) = \sup_{r_{1:n} \in \RR_+, \; w_{1:n} \in \SS_+} \sM_t(\D_n(r_{1:n})|\D_n(w_{1:n})). 
\]
By compactness of~$\RR \times \SS$ and continuity of~$\sM_t(\D_n(r_{1:n})|\D_n(w_{1:n}))$ as a function of~$(r_{1:n}, w_{1:n})$, 
the right-hand side is equal to~$\sM_t(\D_n(r_{1:n}^*)|\D_n(w_{1:n}^*))$ for some~$(r_{1:n}^*, w_{1:n}^*) \in \RR_+ \times \SS_+$. 
Meanwhile, by Theorem~\ref{th:main-result}, the value~$\sM_t(\D_n(r_{1:n}^*)|\D_n(w_{1:n}^*))$  is attained on~$(z_{1:n}, x_{0:n-1}) \in \RR_+ \times \SS_{\pm}$, where~$z_{1:n}^{\vphantom{*}} = r_{1:n}^*$ and~$x_{k-1} = (-1)^{n-k} w_k$ for~$k \in [n]$.
As such,~$\sM_t(\RR_+|\SS_{\pm}) \ge \sM_t(\D_n(r_{1:n}^*)|\D_n(w_{1:n}^*)) = \sM_t(\RR|\SS)$. 
\end{proof}

According to Corollary~\ref{cor:Reinhardt-to-orthant}, for origin-centered Reinhardt domains, we are left with the problem
\begin{equation}
\label{eq:Schur-on-orthant}
\sup_{r_{1:n} \in \RR_+,\; w_{1:n} \in \SS_+} \sF_t(r_{1:n} | w_{1:n}),
\end{equation}
where~$\sF_t(\cdot | w_{1:n}) := \sum_{k \in [n]}  w_k \sch_{(t-n|n-k)}(\cdot)$, defined on~$\R_+^{n}$, is permutation-symmetric~$\forall w_{1:n} \in \R_+^n$. 
Let us now discuss how one can exploit the special structure of~$\sF_t$ to solve~\eqref{eq:Schur-on-orthant} in certain scenarios.

\subsection{Log-convex and log-affine domains of roots.}
Note that the domain~$\RR$ can be assumed symmetric: by the permutation symmetry of~$\sF_t(\cdot|w_{1:n})$,~$\sM_t(\cdot|\SS)$ is not changed if~$\RR$ is replaced with its symmetrization~$\{(z_{\sigma_1},\dots,z_{\sigma_n}):  z_{1:n} \in \RR,\; \sigma \in \mathfrak{G}_n \}$, where~$\mathfrak{G}_n$ is the symmetric group on~$[n]$. 
Note also that~$\sF_t(\cdot | w_{1:n})$ is a polynomial with positive coefficients for any~$w_{1:n} \in \R_+^{n}$. 
In the context of geometric programming~\cite{peterson1980geometric,boyd2007tutorial}, this suggests changing the variables to~$x_{1:n} := \log(r_{1:n})$. 
Now, assume~$\RR_+$ is compact and {\em logarithmically convex}, i.e.~
$
\log(\RR_{+}) := \{\log(r_{1:n}): r_{1:n} \in \RR_+, r_{1:n} > 0\}
$
is a convex set.\footnote{This is a very natural class: each such set is the convergence domain for some power series, and vice versa~\cite{vladimiroff1966methods,hormander1973introduction,shabat1992introduction}.}
Since~$\log\sF_t(\exp(\cdot)|w_{1:n})$ and~$\sup_{w_{1:n} \in \SS_+} \log \sF_t(\exp(\cdot)|w_{1:n})$ are convex and entrywise nondecreasing on~$\R^n$, each of them is maximized at 
an extreme point of~$\log(\RR_+)$. 
If~$\RR_+$ is {\em log-affine}, i.e.~$\log(\RR_+) = \Conv(\boldsymbol{V})$ is a polytope, this leads to the following conceptual method for solving~\eqref{eq:Schur-on-orthant}.

\begin{figure}[H]
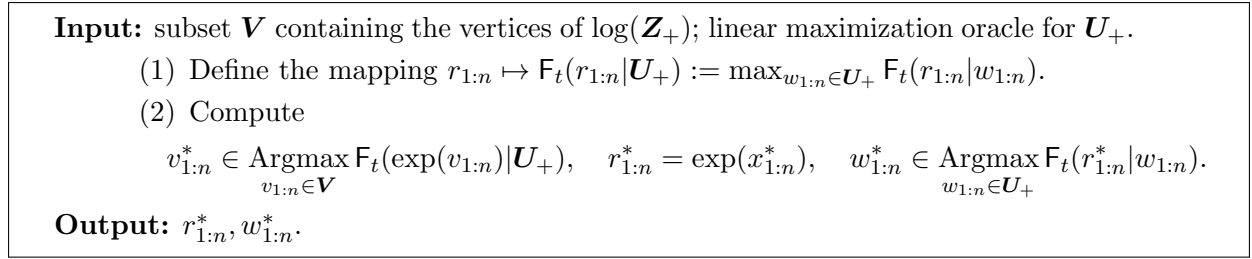

\begin{framed}
\vspace{-0.4cm}
\begin{quote}
\label{eq:Newton-conceptual}
\begin{center}
\end{center}
\noindent
\hspace{-1.0cm}{\bf Input:} subset~$\boldsymbol{V}$ containing the vertices of~$\log(\RR_+)$; linear maximization oracle for~$\SS_+$.
\begin{enumerate}
\vspace{0.1cm}
\item Define the mapping~$r_{1:n} \mapsto \sF_t(r_{1:n}|\SS_+) := \max_{w_{1:n} \in \SS_+} \sF_t(r_{1:n}|w_{1:n})$.
\vspace{0.1cm}
\item Compute
\[
\quad\quad\quad\quad
v_{1:n}^* \in \Argmax_{v_{1:n} \in \boldsymbol{V}} \sF_t(\exp(v_{1:n})|\SS_+), 
\quad r_{1:n}^* = \exp(x_{1:n}^*), \quad
w_{1:n}^* \in \Argmax_{w_{1:n} \in \SS_+} \sF_t(r_{1:n}^*|w_{1:n}).
\]
\vspace{-0.2cm}
\end{enumerate}
\hspace{-1.0cm}{\bf Output:} $r_{1:n}^*, w_{1:n}^*$. 
\end{quote}
\vspace{-0.2cm}
\end{framed}
\vspace{-0.3cm}
\caption{Conceptual procedure for solving~\eqref{eq:Schur-on-orthant} in the case of a log-affine set~$\RR_+$.}
\label{fig:vertex-method}
\end{figure}

\begin{example}
{\em If~$\RR = \D_n(r_{1:n})$, then~$\log(\RR_+) = \{x_{1:n} \in \R^{n}: x_{1:n} \le  \log(r_{1:n})\}$ and
$
\boldsymbol{V} = \{\log(r_{1:n})\}.
$
Therefore,~$\sM_t(\D_n(r_{1:n})|\SS) = \sup_{w_{1:n} \in \SS_+} \sF_t(r_{1:n}|w_{1:n})$ for any origin-centered Reinhardt domain~$\SS$.
}
\end{example}

\begin{example}
{\em Let~$\RR$ be the origin-centered Reinhardt domain with
$
\log(\RR_+) = \{x \in \R^n: Ax \le 0 \}, 
$
where~$A \in \R^{m \times n}$ is of full column rank with nonnegative entries.
The only point in~$\boldsymbol{V}$ is the origin, and computing~$\sM_t(\RR|\SS) = \sup_{w_{1:n} \in \SS_+} \sF_t(\ones_n|w_{1:n})$ reduces to linear maximization over~$\SS_{+}$, cf.~\eqref{eq:hook-content-formula}.  
}
\end{example}

\begin{example}
{\em Let~$\RR$ be the origin-centered Reinhardt domain for~$\RR_+ = \{r_{1:n} \ge \ones_n: \prod_{k \in [n]} r_k \le q\}$.
Then~$\log(\RR_+)$ is a rescaled probability simplex,~$\boldsymbol{V}$ is the set of rescaled canonical basis vectors, and 
\[
\sM_t(\RR|\SS) = \sup_{w_{1:n} \in \SS_+} \sum_{k \in [n]} w_k \sch_{(t-n|n-k)}(q,\ones_{n-1}). 
\]
We can compute~$\sch_{(t|n-k)}(q,\ones_{n-1})$ explicitly, by combining Fact~\ref{fact:hook-bilinear} with the following expressions:
\[
\begin{aligned}
\se_d(q,\ones_{n}) 
	= \binom{n}{d} + q\binom{n}{d-1},
\quad\quad
\sh_d(q,\ones_{n}) 
	= \sum_{k=0}^d  q^{d-k} \binom{n + k - 1}{k}.
\end{aligned}
\]
}
\end{example}

Note that the monotonicities of~$\sF_t(\cdot|w_{1:n})$ and~$\sF_t(\cdot|\SS_+) = \sup_{w_{1:n} \in \SS_+} \sF_t(\cdot|w_{1:n})$ allow to eliminate any vertex~$v \in V$ for which there exists some other~$v' \in V$ with~$v' \ge v$, as in the following example.

\begin{example}
{\em 
If~$\RR$ be origin-centered Reinhardt domain, 
$
\log(\RR_+) = \{ p \ones_n + (1-p)x: \|x\|_{\infty} \le 1 \}
$
with~$p \in (0,1)$. 
Any~$v \in \boldsymbol{V}$ is dominated by~$\ones_n$, thus~$\sM_t(\RR|\SS) = \sup_{w_{1:n} \in \SS_+}  \sum_{k \in [n]} w_k  \sch_{(t-n|n-k)} (e \ones_n)$. 
}
\end{example}

\section{Application to Lagrange interpolation}
\label{sec:shadrin}

When proving Theorem~\ref{th:main-result},  we have expressed the interpolation residuals~$\veps_t(z) = z^{t} - \psi_{t}(z)$ for monomials of degree~$t \in \N_n$
as polynomials in~$z$ that contain the characteristic polynomial as a factor.
We then used these expressions, cf.~\eqref{eq:complete-extraction}, to analyze the derivatives of~$\veps_t(z)$ at~$0$, showing~$\sch$-positivity of these as symmetric polynomials in~$z_{1:n}$. 
More generally, we could consider these derivatives as functions of a complex variable -- in particular, bound the ratio of~$k^{\textup{th}}$ derivatives of~$\veps_t(z)$ and~$\veps_n(z)$ pointwise. Alternatively, we could examine the ratio of the sup-norms of these derivatives on some domain in~$\C$ or~$\R$.
Our focus on these particular criteria is motivated by the progress in approximation theory in the 1990s, regarding sharp error bounds for Lagrange interpolation on~$\R$, which we shall review in a short while.
But before that, let us adapt Proposition~\ref{prop:monomial-interpolation-any-k} to our purposes.

In what follows, we denote the Lagrange polynomial~$L_n[(z_{1:n})^t|z_{1:n}]$ by~$\psi_t(\cdot|z_{1:n})$, the characteristic polynomial~$P_n[z_{1:n}]$ by~$f(\cdot|z_{1:n})$, 
and the Lagrange polynomial~$L_n[g(z_{1:n})|z_{1:n}]$ of~$g(\cdot)$ by~$\hat g(\cdot|z_{1:n})$.


\begin{proposition}
\label{prop:monomial-interpolation-error-via-schur}
Under the premise of Proposition~\ref{prop:monomial-interpolation-0}, for~$t \in \N_n$, any~$0 \le k \le n-1$, and~$z \in \C$, 
\begin{equation}
\label{eq:monomial-interpolation-error-via-schur}
\begin{aligned}
\frac{1}{k!} \frac{d^k}{dz^k} \big[z^{t} - \psi_{t}(z|z_{1:n})\big]
&= (-1)^{n-k} \sum_{d = 0}^{t-n} \binom{t}{n+d} z^{t-n-d}  \sch_{(d|n-k-1)} (z_{1:n}-z1_n) \\
&= z^{t-k} \sum_{d = 0}^{t-n} (-1)^{d} \binom{t}{n+d} \sch_{(d|n-k-1)} (1_n-z^{-1}z_{1:n}) \;\; \text{for} \;\; z \ne 0.
\end{aligned}
\end{equation}
\end{proposition}

\begin{proof}
Notice that the second identity is verified by factoring out~$(-z)$ from each Schur polynomial.
For the first, we use the linearity of the functional~$g \mapsto \frac{d^k}{dz^k}[g(z) - \hat g(z|z_{1:n})]$ and operator~$g \mapsto \hat g(\cdot|z_{1:n})$, and invoke the binomial expansion for~$(z+\zeta)^t$;
this allows to rewrite the left-hand side of~\eqref{eq:monomial-interpolation-error-via-schur} as
\[
\begin{aligned}
\frac{1}{k!} \frac{d^k}{dz^k} \big[z^{t} - \psi_{t}(z|z_{1:n})\big] 
&= \frac{1}{k!} \frac{d^k}{d\zeta^k} \big[(z+\zeta)^{t} - \psi_{t}(z+\zeta|z_{1:n})\big] \big|_{\zeta = 0} \\
&= \frac{1}{k!} \frac{d^k}{d\zeta^k} \big[ (z+\zeta)^{t}  - L_n [(z_{1:n})^{t} | z_{1:n}-z\ones_n](\zeta) \big] \big|_{\zeta = 0} \\
&= \frac{1}{k!} \frac{d^k}{d\zeta^k} \big[ (z+\zeta)^{t}  - L_n [((z_{1:n}-z\ones_n) + z \ones_n)^{t} | z_{1:n}-z\ones_n](\zeta) \big] \big|_{\zeta = 0} \\
&= \frac{1}{k!} \sum_{j=0}^{t} \binom{t}{j} z^{t-j} \frac{d^k}{d\zeta^k} \big[  \zeta^j  - \psi_{j}(\zeta|z_{1:n}-z\ones_n)  \big] \big|_{\zeta = 0} \\
&= -\frac{1}{k!} \sum_{j=n}^{t} \binom{t}{j} z^{t-j} \psi_{j}^{(k)}(0|z_{1:n}-z\ones_n).
\end{aligned}
\]
The result follows by changing the summation index to~$d = j - n$ and invoking Proposition~\ref{prop:monomial-interpolation-any-k}. 
\end{proof}
\noindent 
Let us define the rational multivariate function~$\sQ_{t,n,k}: \C^n \to \C$, symmetric in its arguments, by
\begin{equation*}
\sQ_{t,n,k}(\zeta_{1:n}) 
:= \sum_{d = 0}^{t-n} (-1)^d \frac{\binom{t}{n+d}\sch_{(d|n-k-1)}(\zeta_{1:n})}{\binom{t}{n}\se_{n-k}(\zeta_{1:n})}. 
\end{equation*}
\begin{corollary}
\label{cor:schur-ratio}
Under the premise of Proposition~\ref{prop:monomial-interpolation-error-via-schur}, for any~$z \in \C \setminus \{0\}$ and~$t \in \N_n$, one has
\begin{equation}
\label{eq:schur-ratio}
\frac{[z^{t} - \psi_{t}(z|z_{1:n})]^{(k)}}{\binom{t}{n}  f^{(k)}(z|z_{1:n})}
=  z^{t-n} \sQ_{t,n,k}(\ones_n - z^{-1}z_{1:n}).
\end{equation}
\end{corollary}
Dividing by~$\binom{t}{n}$ corresponds to renormalizing~$z^t$ and~$z^n$ to have the same modulus of~$n^{\textup{th}}$ derivative on~$\T$.
Note also that our choice of the signs in the second identity of~\eqref{eq:monomial-interpolation-error-via-schur}, leading to the {alternating} sum of the ratios of Schur polynomials in~\eqref{eq:schur-ratio}, and prompting that Schur positivity will not be easy to exploit (unless~$z = 0$), is not coincidental.
Indeed, we can put the factor~$(-1)^{d}$ back into each respective Schur function~$\sch_{(d|n-k-1)}$, to get the sum~$c_{d} \sch_{(d|n-k-1)}(1_n-z^{-1}z_{1:n})$ with~$c_{d}$ of the same sign. 
However, we shall be concerned with the case~$\|z_{1:n}\|_{\infty} \le |z|$, where~$z^{-1}z_{1:n}-\ones_n$ have {\em negative} real parts. 
In other words, sign alternation is in the nature of things if one is concerned with~$z = 1$.







\subsection{Context: interpolation on~$\R$.}
Let~$I$ be the segment~$[-1,1]$, and define~$\cF_n(I)$ as the set of~$n$ times weakly differentiable functions~$g: I \to \R$ with~$\|g^{(n)}\|_{I} \le 1$, where~$\|\cdot\|_X$ is the sup-norm on~$X$. 
A classic problem in approximation theory is to characterize the worst-case error of Lagrange interpolation, measuring it uniformly over target~$g \in \cF_n(I)$ and test point~$z \in I$, in~$k^{\textup{th}}$ derivative:
\begin{equation}
\label{eq:interpolation-error-uniform}
\sup_{g \in \cF_n(I)} \|g^{(k)}(\cdot) - \hat g^{(k)}(\cdot|z_{1:n}) \|_{I}, \quad 0 \le k < n.
\end{equation}
The crucial point here is to study this quantity without making strong assumptions on~$z_{1:n}$. 
In the case of~$k = 0$, the folklore result, commonly attributed to Cauchy, is that for arbitrary interpolation grid and~$k = 0$,~the supremum is attained on the monomial~$\frac{1}{n!} z^{n}$, and thus is equal to~$\tfrac{1}{n!} \|f(\cdot|z_{1:n})\|_{I}$.\footnote{Note that the segment~$I$ and the bound on~$\|g^{(n)}\|_{I}$ can be made arbitrary, by rescaling~$g$ and shifting/rescaling~$z$.}
Since this supremum trivially reduces to~$\frac{1}{n!} \|f^{(n)}(\cdot|z_{1:n})\|_I = 1$ for~$k = n$, one might speculate that
\begin{equation}
\label{eq:shadrin-bound-uniform}
\sup_{g \in \cF_n(I)} \|g^{(k)}(\cdot) - \hat g^{(k)}(\cdot|z_{1:n}) \|_{I} = \frac{1}{n!} \|f^{(k)}(\cdot|z_{1:n})\|_{I}, \quad  0 \le k \le n-1.
\end{equation}
This was first conjectured in the 1990s by Kallioniemi~\cite{kallioniemi1990landau}, then proved by Howell~\cite{howell1991derivative} in the special case of~$k = n-1$.
The way more challenging general case was eventually resolved by Shadrin~\cite{shadrin1995error}. 

One should be cautioned against dismissing~\eqref{eq:shadrin-bound-uniform}  as a simple consequence of Bernstein's inequality; in fact, it is a much more delicate result, closely related to the Landau-Kolmogorov problem~\cite{shadrin2004twelve}.
To grasp the nature of Shadrin's result, assume that~$z_{1:n}$ is simple and sorted in the increasing order.
Define~$a_{1:n-k-1}$ and~$b_{1:n-k-1}$, respectively, as the roots of~$[\prod_{j \in [n] \setminus 1} (z-z_j)]^{(k)}$ and~$[\prod_{j \in [n-1]} (z-z_j)]^{(k)}$, sorted in the same way, and complete these lists with~$a_0 = -1$,~$a_{n-k} = z_n$ and~$b_0 = z_1$,~$b_{n-k}=1$. 
Define two sets,~$\cI_{n,k}(z_{1:n}) := \bigcup_{s = 0}^{n-k} I_s$ and~$\cJ_{n,k}(z_{1:n}) := \bigcup_{s = 0}^{n-k-1} J_s$, with~$I_s := [a_s,b_s]$ and~$J_s := (b_{s},a_{s+1})$. 
In~\cite{kallioniemi1976bounds}, Kallioniemi showed that~$a_{s} \le b_s \le a_{s+1} \le b_{s+1}$ for~$0 \le s < n-k$; in other words,~the sequence~$I_0, J_0, I_1, \dots, J_{n-k-1}, I_{n-k}$ partitions~$I$ with all segments listed from left to right. Moreover, he showed that over the segments~$I_s$, the monomial~$\frac{1}{n!} z^{n}$ is {\em pointwise} worst-case:
\begin{equation}
\label{eq:shadrin-bound-pointwise}
\sup_{g \in \cF_n(I)} |g^{(k)}(z) - \hat g^{(k)}(z|z_{1:n})| = \frac{1}{n!} |f^{(k)}(z|z_{1:n})| \quad \forall z \in \cI_{n,k}(z_{1:n}).
\end{equation}
In~\cite{kallioniemi1990landau}, he conjectured that in each~$J_s$,~$\sup_{g \in \cF_n(I)} |g^{(k)}(\cdot) - \hat g^{(k)}(\cdot|z_{1:n})|$ is maximized at an endpoint,
which implies~\eqref{eq:shadrin-bound-uniform}, and that was what Shadrin later proved in~\cite{shadrin1995error}.
Shadrin's proof was based on the characterization of extremal functions in~\eqref{eq:interpolation-error-uniform} as {perfect splines}, which were earlier introduced by Schoenberg~\cite{schoenberg1964best} and Sard~\cite{sard1963linear,sard1967optimal} in the context of other extremal problems on~$\R$, such as the Landau-Kolmogorov problem and the Stechkin problem, and developed in the 1970s (e.g.~\cite{pinkus1978some,goodman1978another}).
Shadrin's subsequent expository article~\cite{shadrin2004twelve} explored the various connections between these problems.
\vspace{0.2cm}

\subsection{Interpolation with self-conjugate nodes.}
In the remainder of this section, we consider the following setup:~$z \in I$ as before, but the nodes are now complex and come in conjugate pairs, so that the arising symmetric polynomials are real-valued (in particular, the characteristic polynomial has real coefficients).
We formulate some conjectures on the worst-case interpolation error -- first in the approximation-theoretic language, then in terms of Schur polynomials, using the connection given by Proposition~\ref{prop:monomial-interpolation-error-via-schur} and the reduction to monomials that shall be justified in Section~\ref{sec:monomial-reduction}.

From now on, and unless stated otherwise, the interpolation nodes are assumed to belong to~$\D$.
Accordingly, we change the target class to~$\cE_n(\D)$, the set of all {\em entire} functions~$g: \C \to \C$ that satisfy~$\|g^{(n)}\|_\D \le 1$.
The motivation for this choice is twofold. 
On the one hand,~$\mathscr{E}_n(\D)$ includes polynomials, and therefore its restriction onto~$I$ is uniformly dense in~$\cF_n(I)$ by the Stone-Weierstrass theorem.
On the other hand, in Section~\ref{sec:monomial-reduction} we show that monomials are the only worst-case functions in~$\mathscr{E}_n(\D)$, with respect to both the pointwise criterion (cf.~\eqref{eq:shadrin-bound-pointwise})
and the uniform criterion, (cf.~\eqref{eq:interpolation-error-uniform}).
As the result, Proposition~\ref{prop:monomial-interpolation-error-via-schur} and Corollary~\ref{cor:schur-ratio} allow to reformulate these conjectures as statements about the behavior of the alternating sum of Schur polynomials (cf.~\eqref{eq:monomial-interpolation-error-via-schur}) or of their ratios (cf.~\eqref{eq:schur-ratio}).


To facilitate the discussion, let us first introduce the relevant notion from~\cite{ellard2020families}.
\begin{definition}
\label{def:self-conjugate}
A list~$z_{1:n} \in \C^n$ is {\em self-conjugate} if
$(i)$ for any~$z \in \{z_1,\dots,z_n\}$ with~$\Im(z) \ne 0$, the conjugate~$\bar z~$ is also contained in~$\{z_1,\dots,z_n\}$;
$(ii)$ $z_{1:n}$ contains even number of copies of each~$z \in \R$.
\end{definition}
%
At first glance, requirement~$(ii)$ might appear redundant. 
However, it proves to be crucial in~\cite{ellard2020families}, whose techniques seem to be relevant in the context of the conjectures to be discussed.
Note that~$z_{1:n}$ is self-conjugate if and only if~$f(\cdot|z_{1:n})$ has real coefficients and even multiplicities of real roots.

The first conjecture concerns the worst-case ratio of residuals, cf.~\eqref{eq:shadrin-bound-pointwise}, evaluated at the points~$\{0,1\}$. 
\begin{conjecture}
\label{conj:pointwise-self-conjugate-z1}
For any~$0 \le k \le n-1$ and arbitrary self-conjugate grid~$z_{1:n} \in \D^n$, it holds that
\begin{equation*}
\sup_{g \in \cE_n(\D)} |g^{(k)}(z) - \hat g^{(k)}(z|z_{1:n})| \bigg|_{z = 1} = \frac{1}{n!} |f^{(k)}(z|z_{1:n})| \bigg|_{z = 1}.
\end{equation*}
Moreover, under the additional assumption that~$\Re(z_{1:n}) \in \R^n_+$, this equation is also valid at~$z = 0$.
\end{conjecture}

\noindent
By our previous remarks, in order to prove---or disprove---this conjecture, it suffices to consider the normalized monomials~$\frac{(m-n)!}{m!} z^m$,~$m \ge n$. 
As such, Corollary~\ref{cor:schur-ratio} provides its equivalent formulation:

\begin{conjecture}[Equivalent to~Conjecture~\ref{conj:pointwise-self-conjugate-z1}]
\label{conj:pointwise-self-conjugate-z1-schur}
For all~$0 \le k < n \le t$ and self-conjugate~$z_{1:n} \in \D^n$,
\begin{equation}
\label{eq:pointwise-self-conjugate-z1-schur}
|\sQ_{t,n,k}(z_{1:n} + \ones_n)| \le 1.
\end{equation}
Moreover, if the grid additionally satisfies~$\Re(z_{1:n}) \in \R^n_+$, then~$|\sch_{(t-n|n-k-1)}(z_{1:n})| \le \binom{t}{n} |\se_{n-k}(z_{1:n})|$.
\end{conjecture}

Note that both inequalities of Conjecture~\ref{conj:pointwise-self-conjugate-z1-schur} are trivially satisfied---with equality---when~$t = n$. 
Our belief in them is justified by the following considerations.

In the case of~$z = 0$, we are motivated by~Khare and Tao~\cite{khare2021sign}, who showed that for any partitions~$\lambda \contains \mu$, the ratio
$
{\sch_{\lambda}(x_{1:n})}/{\sch_{\mu}(x_{1:n})}
$
is coordinatewise nondecreasing on~$\R^n_+$. 
Applying this for~$x_{1:n} \in I^n$ and our pair of partitions, we get
\[
\frac{\sch_{(t-n|n-k-1)}(x_{1:n})}{\binom{t}{n}\se_{n-k}(x_{1:n})} 
\le \frac{\sch_{(t-n|n-k-1)}(\ones_n)}{\binom{t}{n}\se_{n-k}(\ones_n)} 
\overset{\eqref{eq:hook-content-formula}}{=} 
\frac{\binom{t}{k} \binom{t-k-1}{t-n}}{\binom{t}{n}\binom{n}{k}}
= \frac{n-k}{t-k} \le 1,
\]
and recover the second claim of Conjecture~\ref{conj:pointwise-self-conjugate-z1} in the real case (it is also recovered from~\eqref{eq:shadrin-bound-pointwise}, replacing~$I$ with~$[0,1]$).
While the authors of~\cite{khare2021sign} note that their results do not seem to extend from~$\R_+$ to~$\D$, their arguments do not seem to exclude the right half-disc. One should also keep in mind that for the second claim of Conjecture~\ref{conj:pointwise-self-conjugate-z1}, we only need very special partitions. In particular, it would suffice to show that ratios of the form
${\sch_{(a|b-1)}(z_{1:n})}/{\se_{b}(z_{1:n})}$ are coordinatewise nondecreasing in~$\Re(z_{1:n})$ and nonincreasing in~$\Im(z_{1:n})$, for self-conjugate~$z_{1:n} \in \D^n$ with nonnegative real parts. 
The intution here is that the odd powers of~$\Im(z_1),...,\Im(z_n)$ cancel, and the even ones are negative.

In the case of~$z = 1$, we make the following observations. First, the shifted nodes~$\zeta_{1:n}$ live in the disc~$\{ \zeta \in \C: |\zeta - 1| \le 1 \}$ in the right half-plane, so there is no need to require~$\Re(z_{1:n}) \in \R^n_+$.
Second, one can verify~\eqref{eq:pointwise-self-conjugate-z1-schur} in the easiest case~$t = n+1$ with even~$k$. Here, the inequality becomes
\[
0 \le \se_1(\zeta_{1:n}) - \frac{\se_{n-k+1}(\zeta_{1:n})}{\se_{n-k}(\zeta_{1:n})} \le 2n+2.
\]
Clearly,~$\se_d(\zeta_{1:n}) \ge 0$ for any~$d \in \N_1$, and the upper bound follows since~$\se_{1}(\zeta_{1:n}) \le 2n$.
For the lower bound, since~$n-k$ is even,~\cite[Theorem~2.2(1)]{ellard2020families} gives~$\se_{1} \se_{n-k} - \se_{n-k+1} = \se_{1} \se_{n-k} - \se_0 \se_{n-k+1} \ge 0$.
(Notice that we cannot simply invoke~Schur positivity here, since the variables are not real-valued.)
Finally, we note that~\cite{bergeron2004some,lam2006schur} prove the positivity of some differences of products of Schur functions.

More broadly, of interest is the structure of the ``Kallioniemi set," defined by analogy with~\eqref{eq:shadrin-bound-pointwise} as
\begin{equation}
\label{eq:kallioniemi-set}
\mathcal{\cI}_{n,k}^{\D}(z_{1:n}) := 
\bigg\{ z \in I: \sup_{g \in \cE_n(\D)} |g^{(k)}(z) - \hat g^{(k)}(z|z_{1:n})| = \frac{1}{n!} |f^{(k)}(z|z_{1:n})| \bigg\},
\end{equation}
for self-conjugate grids~$z_{1:n} \in \D^n$. It seems likely that the structure identified by Kallioniemi and Shadrin in the real setup is preserved -- in other words, this set consists of
nonoverlapping segments, and~$z \mapsto |[g(z) - \hat g(z|z_{1:n})]^{\vspace{-0.1cm}(k)}|$ is maximized at an endpoint of each complementary segment.
This would imply that Shadrin's result generalizes to self-conjugate grids in~$\D^n$, prompting the following
\begin{conjecture}
\label{conj:uniform-self-conjugate}
For any~$0 \le k \le n-1$ and arbitrary self-conjugate grid~$z_{1:n} \in \D^n$, it holds that
\begin{equation}
\label{eq:uniform-self-conjugate}
\sup_{g \in \cE_n(\D)} \| g^{(k)}(\cdot) - \hat g^{(k)}(\cdot|z_{1:n}) \|_{I}
= \frac{1}{n!} \|f^{(k)}(\cdot|z_{1:n})\|_{I} \, ,
\end{equation}
or equivalently in terms of Schur polynomials:
\[
\sup_{t \in \N_n} \left\{ {\binom{t}{n}}^{-1} \max_{z \in I} \left| \sum_{d = 0}^{t-n} {\binom{t}{n+d}} z^{t-n-d}  \sch_{(d|n-k-1)} (z_{1:n}-z\ones_n) \right| \right\} = \max_{z \in I} |\se_{n-k}(z_{1:n}-z\ones_n)| \, .
\]
\end{conjecture}
\noindent
Here, one must study the sum in the left-hand side analytically as a polynomial in~$z \in I$, 
possibly using the result of~Weigandt~\cite{weigandt2023derivatives,grinberg2024diagonal} that expresses~$\sum_{j \in [n]} \frac{\partial}{\partial x_j} \sch_{\lambda}(x_{1:n})$ in terms of Schur polynomials. 

\subsubsection*{Palindromic case.}
We now briefly discuss the case~$z_{1:n} \in \T^n$. Since~$\bar z = 1/z$ on~$\T$, polynomial~$f(z)$ coincides with its reciprocal~$z^{n}f(1/z)$, i.e.~is {\em palindromic}: $f_{k} = f_{n-k}$ for~$0 \le k \le n$. Equivalently, the dynamical system defined in~\eqref{def:recurrence} is time-reversible. 
Of possible relevance here is the result in~\cite{alexandersson2021symmetric}, recasting~$\sch_{\lambda}(x_1,\dots,x_n,1/x_1,\dots,1/x_n)$ as a Schur polynomial in the variables~$x_j + 1/x_j$. Invoking we see that~$\Im(z_{1:n})$ vanish here, along with the difficulties they used to cause. On the flip side, the new parametrization might complicate the nice algebraic structure associated with hook partitions.

\subsection{Reduction to monomials.}
\label{sec:monomial-reduction}

We now reduce the suprema over~$\cE_{n}(\D)$ in~\eqref{eq:kallioniemi-set}--\eqref{eq:uniform-self-conjugate} to those over monomials. 
Interchanging~$\sup_{g \in \cE_n(\D)}$ and~$\sup_{z \in I}$ in~\eqref{eq:uniform-self-conjugate}, it suffices to treat the pointwise case.

\begin{proposition}
\label{prop:extreme-points}
For~$z \in \C$ and simple~$z_{1:n} \in \C^n,$ the functional~$g \in \cE_n(\D) \mapsto |[g(z) - \hat g(z|z_{1:n})]{}^{(k)}|$ is maximized on a monomial of the form~$\frac{(m-n)!}{m!} z^m$,~$m \ge n$.
\end{proposition}

\begin{proof}
For any~$z \in \C$ and simple~$z_{1:n} \in \C^n$, the above functional is convex and continuous on~$\cE_n(\D)$ (with respect to the uniform metric on~$\D$).
Now, let~$\cH_n(\D)$ be the set of functions holomorphic on~$\D$ with~$\|g^{(n)}\|_\D \le 1$. 
By Montel's theorem,~$\cH_n(\D)$ is precompact; hence~$\cE_n(\D)$ is compact as its closed subset.
By the Krein-Milman theorem and Bauer's maximum principle, it suffices to show that all extreme points of~$\cE_n(\D)$ are of the form~$\frac{(m-n)!}{m!} a z^{m} + p(z)$, where~$a \in \T$ and~$p(z)$ is a polynomial of degree less than~$n$.
Considering~$[\cdot]^{(k)}$ as a linear operator, we reduce arbitrary~$n \in \N_0$ to~$n = 0$, in which case the claim follows from Lemma~\ref{lem:extreme-points} below.
\end{proof}

\begin{lemma}
\label{lem:extreme-points}
{Extreme points of the set of entire functions with modulus~$\le 1$ on~$\T$ are monomials.} 
\end{lemma}
\begin{proof}
To be an extreme point of~$\cE_0(\D)$,~$g$ has to satisfy~$\|g\|_{\D} = 1$, i.e.~have a nonempty support set
\[
\Supp(g) := \{z \in \T: |g(z)| = 1 \}.
\] 
A classical result~(e.g.~\cite[II.I.3]{valiron1949lectures}) is that for any entire function~$g$,~$\Supp(g)$ is either the whole~$\T$ or a finite set.
In the former case,~$g$ must be a monomial. Indeed,~$g_0(z) := 1/\overline{g(1/\overline{z})}$ is meromorphic on~$\C \setminus \{0\}$ and coincides with~$g$ on~$\T$, so by analytic continuation~$g(z) = g_0(z)$ for all~$z \ne 0$.
Moreover,~$g_0$ has the pole of order~$m$ at~$0$, where~$m \ge 0$ is the order of~$g$ at~$0$. 
Hence~$g$ is a polynomial of degree~$m$ without lower-degree terms, i.e. a monomial of the form~$cz^m$, as required.

In the finite case, assume~$\Supp(g) = \{\zeta_1, \dots, \zeta_d\}$. 
Note that~$s(z) := 1-|g(z)|^2$ is real analytic, nonnegative on~$\D$, and has zeroes~$\zeta_1, \dots, \zeta_d$ on~$\T$; let~$m_1, \dots, m_d$ be its vanishing orders at these zeroes. 
Note that~$r(z) := 1-|g(z)|$ has the same zeroes and vanishing orders, and~$r(z) \sim \frac{1}{2} s(z)$ as~$z \to \zeta_j$. 
Now, define a polynomial~$h(z) := \prod_{j \in [d]} (z-\zeta_j)^{m_j}$ that vanishes to order~$m_j$ at~$\zeta_j$.
Clearly, there are positive constants~$a_j, b_j$ such that, in a neighborhood~$U_\delta(\zeta_j)$ of each~$\zeta_j$, one has~$r(z) \ge a_j |z-\zeta_j|^{m_j}$ and~$|h(z)| \sim b_j|z-\zeta_j|$; therefore,~$r(z)/|h(z)| \ge c > 0$ on~$U_{\delta} := \bigcup_{j \in [d]} U_{\delta}(\zeta_j)$. 
Meanwhile, on the compact set~$\D \setminus U_{\delta}$, one has~$r(z) \ge \alpha > 0$ and~$|h(z)| \ge \beta > 0$ by continuity. 
Thus, for some~$\veps > 0$ one has~$r(z) \ge \veps |h(z)|$ and
$
|g(z) \pm \veps h(z)| \le |g(z)| + \veps |h(z)| \le |g(z)| + r(z) = 1
$
on~$\D$, which contradicts the premise that~$g(\cdot)$ is an extreme point of~$\cE_0(\D)$.
\end{proof}


\begin{remark}
\label{rem:dyakonov}
{\em Note that in the proof of Proposition~\ref{prop:extreme-points}, we cannot replace~$\cE_n(\D)$ with its subset~$\cE_{n,m}(\D)$ containing only polynomials of degree at most~$m = n+d$, for any finite~$d \in \N_0$.
To see why, note that this would lead us to replacing~$\cE_{0}(\D)$ with~$\cE_{0,d}(\D)$ in Lemma~\ref{lem:extreme-points}. However, the perturbation polynomial~$h(z)$ is of degree~$\sum_{j \in [d]} m_j$, and~$s(z)$ is a nonnegative trigonometric polynomial of degree~$2d$ on~$\T$.
Therefore,~$\deg(h)$ is potentially as large as~$2d$, and the perturbed polynomials~$g \pm h$ might be of degree larger than~$d$. 
And indeed, the extreme set of~$\cE_{0,d}(\D)$, characterized in~\cite{dyakonov2003extreme}, contains not only monomials.
Thus,  in view of the conjectured optimality of the {\em lowest-degree} monomial~$\frac{1}{n!}z^n$, even if one is interested in the problem of maximizing~$g \mapsto |[g(z) - \hat g_n(z|z_{1:n})]^{(k)}|$ over~$\cE_{n,m}(\D)$ for a fixed~$m$, it is natural to ``lift'' it to~$\cE_n(\D)$, the closed convex hull of~$\bigcup_{m \in \N_n}\cE_{n,m}(\D)$.
}
\end{remark}

\bibliography{references}
\bibliographystyle{unsrt}

\end{document}